\documentclass[english,11pt]{article}
\usepackage[utf8]{inputenc}
\usepackage{amssymb,amsmath,amsfonts,amsthm,rotating,setspace,graphicx,float,booktabs,caption}
\usepackage{amssymb,amsmath,amsthm,rotating,setspace,verbatim}
\usepackage{enumerate}
\usepackage{multirow}
\usepackage{lscape}
\usepackage{tikz,pifont}
\usepackage{esint}
\usepackage{mathrsfs}
\usepackage{geometry,url}
\usepackage{bbm}
\usepackage{paralist}
\usepackage[hidelinks]{hyperref}
\usepackage[title]{appendix}
\usepackage{latexsym}
\usepackage{epsfig}
\usepackage{color}

\usepackage{authblk,  empheq}

\newtheorem{thm}{Theorem}[section]
\newtheorem{lemma}[thm]{Lemma}

\newtheorem{cor}[thm]{Corollary}

\theoremstyle{definition}
\theoremstyle{definition}

\numberwithin{equation}{section}

\def\K1{K_{\alpha, \beta}}
\def\Bo{B_R}

\def\B2x{B_{2R}(x)}

\def\xR{\left(\frac{x}{R}\right)}

\def\ds{\displaystyle}

\def\om{\Omega}
\def\R{\mathbbm{R}}

\newcommand{\abs}[1]{\lvert#1\rvert}

\title{Non-homogeneous fourth order elliptic inequalities with the convolution term}
\author[ ]{Zhe Yu\footnote{ORCID: 0000-0002-2261-4981} }
\affil[ ]{School of Mathematics and Statistics}
\affil[ ]{University College Dublin }
\affil[ ]{Belfield, Dublin 4, Ireland}
\affil[ ]{E-mail: {\tt zhe.yu@ucdconnect.ie}}

\date{}
\begin{document}

\maketitle

\begin{abstract}
\vspace{0.3cm} 
\noindent We are concerned with the study of the twin non-local inequalities featuring non-homogeneous differential operators
$$\displaystyle -\Delta^2 u + \lambda\Delta u \geq (K_{\alpha, \beta} * u^p)u^q \quad\text{ in } \mathbbm{R}^N (N\geq 1),$$ 
and
$$\displaystyle \;\;\,\Delta^2 u - \lambda\Delta u \geq (K_{\alpha, \beta} * u^p)u^q \quad\text{ in } \mathbbm{R}^N (N\geq 1),$$
with parameters $\lambda, p, q >0$,  $0\leq \alpha \leq N$ and $\beta>\alpha-N$. In the above inequalities the potential $K_{\alpha,\beta}$ is given by $K_{\alpha, \beta}(x) = |x|^{-\alpha}\log^{\beta}(1 + |x|)$ while $K_{\alpha, \beta} * u^p$ denotes the standard convolution operator in $\mathbbm{R}^N$. We discuss the existence and non-existence of non-negative solutions in terms of $N, p, q, \lambda, \alpha$ and $\beta$.
\end{abstract}


\medskip


\newpage

\section{Introduction}
This paper is concerned with the study of non-negative solutions to the following inequalities involving the bi-harmonic operator and with the convolution term:
\begin{equation}\label{P-}
-\Delta^2 u + \lambda \Delta u \geq (\K1 * u^p)u^q \quad\text{in }\;\;\R^N (N \geq 1),\tag{$P^-$}
\end{equation}
and 
\begin{equation}\label{P+}
\;\;\,\Delta^2 u - \lambda \Delta u \geq (\K1 * u^p)u^q \quad\text{in }\;\;\R^N (N \geq 1).\tag{$P^+$}
\end{equation}
The above inequalities assume that $N \geq 1$, $\lambda, p, q > 0$ and 
\begin{equation}\label{ptt}
\K1(x) = \abs{x}^{-\alpha}\log^{\beta}(1 + \abs{x}), \quad\text{where } 0 \leq \alpha \leq N,\, \beta > \alpha - N.
\end{equation} 
We note that 
$\K1(x) \in L^1_{loc}(\R^N)$ and 
$$
\begin{aligned}
&\K1(x) \simeq |x|^{-\alpha + \beta}\mbox{ as } |x| \to 0,\\[0.2cm]
&\K1(x) \simeq |x|^{-\alpha}\log^{\beta} |x| \mbox{ as } |x| \to \infty.
\end{aligned}
$$
The quantity $\K1 * u^p$ represents the convolution operation in $\R^N$ given by
\[
(\K1 \ast u^p)(x)=\int_{\R^N} \frac{u^p(y)\log^{\beta}(1 + \abs{x - y})}{\abs{x-y}^{\alpha}} dy \quad \text{for all } x \in \R^N.
\]
This paper will examine classical non-negative solutions of \eqref{P-} (respectively of \eqref{P+}), that is, functions $u \in  \mathcal{C}^4(\R^N)$ such that 
\begin{equation}\label{KKK}
u(x)\geq 0\;, \;\; \big(\K1 * u^p\big)(x)<\infty\quad\mbox{for all }\; x\in \R^N,
\end{equation}
and $u$ satisfies \eqref{P-} (respectively of \eqref{P+}) pointwise in $\R^N$.

For $\alpha \in (0, N)$ and $\beta = 0$ the potential $K_{\alpha, 0}$ becomes the standard Riesz potential which plays a crucial role in potential theory and partial differential equations. Beyond this setting, the case $\alpha = N$ is related to the exponential transform of a domain as introduced in \cite{GP03} (see also \cite{P03}). Precisely, given a domain $D \subset \R^N$, we set
\[E_D(x) = \exp\left\{-\frac{2}{\abs{\partial B_1}}\int_C K_{N, 0} (x- y) dy\right\} \text{ for all } x \in \R^N \setminus \overline{D}. \]
By polarization, this definition is extended to points $x \in D$ by
\[H_D(x) = \lim\limits_{R \to \infty} \frac{1}{R^2E_{{B_R}\setminus D(x)}}.\]
We ascertain from \cite{GP03} that $E_\om$ is superharmonic in $\R^N \setminus D$, while $H_D$ is subharmonic in $D$. Also, $H_D$ has explicit expression for common domains $D$ whose boundary is described by quadratic polynomials. In particular, when $D$ is the unit ball $B_1$ in $\R^N$, one has $H_{B_1} = (1 - \abs{x}^2)^{-1}$. The exponential transform $E_D$ was used in \cite{GP98} to study the regularity of free boundaries in two dimensions. Furthermore, potentials of type $K_{N,\beta}$ are related to logarithmic Besov spaces \cite{CDT16, LYZ21}.

The study of various non-homogeneous higher order differential operators has attracted a significant attention in recent years. Results involving the operator $\Delta^2 + \Delta$ can be found in \cite{FJMM22, LW23, MO24}. In the present work, we study a class of differential inequalities involving the operators $-\Delta^2 + \lambda \Delta$ and $\Delta^2 - \lambda \Delta$, $\lambda>0$.  The operator $\Delta^2 u - \lambda\Delta u$ appears in Mechanics, in the study of suspension bridges in the early 20th century. In dimension $N = 1$, we have
\[\Delta^2 u - \lambda \Delta u = u_{xxxx} - \lambda u_{xx},\]
and the suspension bridges equation reads
\begin{equation}\label{Melan}
u_{xxxx} - (H + h(u)) u_{xx} + \frac{q}{H}h(u) = p(x),\quad \text{where } 0 \leq x \leq L.
\end{equation}
In above equation, $L>0$ is the length of the beam, $q$ and $p(x)$ are the dead and live loads per unit length applied to the beam, and the non-local term $h(u)$ is the additional tension in the cable produced by the live load $p$. The equation \eqref{Melan} is called the Melan equation \cite{Melan} in suspension bridges theory introduced in 1913  (see also \cite{BFG16, F22, GS18} for further results on this topic as well as \cite{A22, GS21} for suspension bridge equation with non-local terms). In 1980s, the authors demonstrated that the nonlinear feature of such non-homogeneous operator provided a good fit in the study of travelling waves in suspension bridges in \cite{LM90} and \cite{MW87}. Later, in \cite{LM94, MP97}, the problem with a more general nonlinear non-negative term $f(x, u)$, such that
\[\left\{\begin{aligned}
&\Delta^2 u - \lambda \Delta u = f(x, u)&\quad \text{in }&\Omega,\\
&\Delta u = 0, u = 0&\quad \text{on }&\partial \Omega,
\end{aligned}\right.\]
was investigated in the case where $\Omega$ is smooth and bounded. Ground state solutions of the above equation were discussed in the recent work \cite{FS22}.

Our approach is based on:
\begin{itemize}
    \item Obtaining local properties of solutions. We show that any non-negative solution $u$ of \eqref{P-} satisfies $\Delta u\geq 0$ in $\R^N$ and any non-negative solution $u$ of \eqref{P+} satisfies $-\Delta u\geq 0$ in $\R^N$.
    \item We use Harnack inequalities \cite{Tru67} in order to derive estimates for the non-existence of a solution.
    \item We provide further integral estimates involving logarithmic terms (see Lemma \ref{lm1} below) that extends the previous results in the literature (see \cite{BP01, CZ23, DMMO24, F1, F2, F3, GKS20, MGZ23, MS13, QT21}).
\end{itemize}
It turns out that the presence of the logarithmic term in \eqref{P+} leads to different conditions for the existence of a non-negative solution. 

We first discuss the non-existence of non-negative solutions to \eqref{P-}. To the best of our knowledge, this is the first result concerning the differential operator  $-\Delta^2 + \lambda \Delta$.

\begin{thm}\label{thm1} {\rm (Non-existence for \eqref{P-})}\\
Assume that $0 < \alpha < N$, $p, q > 0$ and $\lambda > 0$. Then \eqref{P-} has no non-negative solutions in $\R^N$ if one of following conditions hold: 
\begin{itemize}
\item[\rm (i)] $p \geq 1$; 
\item[\rm (ii)] $p < 1$ and $u$ is bounded;
\item[\rm (iii)] $u$ is radial.
\end{itemize}
\end{thm}
From Theorem \ref{thm1}, we notice that in most of the cases related to $p\geq 1$ and $p<1$, the inequality \eqref{P-} has no non-negative solutions. 

The corresponding non-existence part for \eqref{P+} is presented in our next result below. 

\begin{thm}\label{thm2} {\rm (Non-existence for \eqref{P+})}\\
Assume $p, q > 0$ and $\lambda > 0$. Then \eqref{P+} has no non-negative solutions in $\R^N$ if one of following conditions hold: 
\begin{enumerate}[\rm(i)]
    \item $1 \leq N \leq 2;$ 
    \item $N \geq 3$,\quad$0\leq \alpha \leq 2\quad$and\quad $1 \leq p  < \frac{N-\alpha}{N - 2};$
    \item $N \geq 3$,\quad$0\leq \alpha \leq 2$,\quad$p = \frac{N-\alpha}{N - 2}\quad$and\quad$\beta \geq -1;$
    \item $N \geq 3$,\quad$p \geq 1$,\quad$p + q < \frac{2N - \alpha}{N - 2}\quad$and\quad$\beta > \alpha - N;$
    \item $N \geq 3$,\quad$p \geq 1$,\quad$p + q = \frac{2N - \alpha}{N - 2}\quad$and\quad$\beta > \frac{1}{p + q} - 1;$
    \item $N \geq 3$,\quad$0\leq \alpha < 2\quad$and\quad$1 < q  < \frac{N - \alpha}{N - 2};$
    \item $N \geq 3$,\quad$0\leq \alpha < 2$,\quad$q  = \frac{N - \alpha}{N - 2}\quad$and\quad$\beta > \frac{1}{q}-1;$
    \item $N \geq 3$,\quad$0\leq \alpha < 2$,\quad $p=\frac{N}{N-2}$,\quad$q  = \frac{N - \alpha}{N - 2}\quad$and\quad$\beta > -2 + \frac{1}{q};$
    \item $N \geq 3$,\quad$0\leq \alpha < 2$,\quad $p=\frac{N - \alpha}{N-2}$,\quad$q  = \frac{N}{N - 2}\quad$and\quad$\beta > -2 + \frac{1}{q}.$
\end{enumerate}
\end{thm}

We next turn to existence results of non-negative solutions to \eqref{P+}. Under the assumption of $N \geq 3$ and $\alpha, \beta$ satisfy \eqref{ptt}, we provide essentially optimal conditions in terms of $p, q, \alpha, \beta$ and $\lambda$ for the existence of a non-negative solution. Precisely, we have:
\begin{thm}\label{thm3} {\rm (Existence for \eqref{P+})}\\
Assume $0 \leq \alpha < N$, $\beta > \alpha - N$, $p > 0$ and $q > 0$. Then, there exists $\lambda > 0$ so that \eqref{P+} has non-negative solutions in $\R^N$ if one of following conditions hold: 
\begin{enumerate}[\rm(i)]
    \item $p > \frac{N - \alpha}{N - 2}$,\quad$q > \frac{N - \alpha}{N - 2}$,\quad$p + q > \frac{2N - \alpha}{N - 2}\quad$and\quad$\beta > \alpha - N;$
    \item $p = \frac{N - \alpha}{N - 2}$,\quad$q > \frac{N}{N - 2}\quad$and\quad$\beta < -1;$
    \item $p > \frac{N}{N - 2}$,\quad$q = \frac{N - \alpha}{N - 2}\quad$and\quad$\beta < -1$;
    \item $p > \frac{N - \alpha}{N - 2}$,\quad$q > \frac{N - \alpha}{N - 2}$,\quad$p + q = \frac{2N - \alpha}{N - 2}\quad$and\quad$\beta < -1$;
    \item $p = \frac{N - \alpha}{N - 2}$,\quad$q = \frac{N}{N - 2}\quad$and\quad$\beta < -2$;
    \item $p = \frac{N}{N - 2}$,\quad$q = \frac{N - \alpha}{N - 2}\quad$and\quad$\beta < -2.$
\end{enumerate}
\end{thm}

The outcome of Theorem \ref{thm2} and Theorem \ref{thm3} can be extended to the extremal case $\alpha = N$. In this setting, condition \eqref{ptt} yields $\beta>0$. We find that Theorem \ref{thm3}(i) provides the optimal condition for the existence of non-negative solutions to \eqref{P+}.  More exactly, we have
\begin{thm}\label{thm4} Assume $N \geq 3$, $p \geq 1$, $q > 0$, $\alpha = N$ and $\beta > 0$. There exists $\lambda>0$ so that \eqref{P+} has a non-negative solution in $\R^N$ if and only if
$$
p + q > \frac{N}{N - 2}.
$$
\end{thm}
Moreover, for $p,\,q\geq 1$, the results in Theorem \ref{thm2} and Theorem \ref{thm3} are nearly optimal. Precisely, we conclude that conditions (i)-(vi) in Theorem \ref{thm3} are optimal for the existence of a non-negative solution if $0 \leq \alpha < N$ and $\beta < -2$ or $\beta > -1 + \frac{1}{q}$. 
\begin{cor}\label{cor}
Assume $N \geq 3$, $p,q \geq 1$. 
\begin{enumerate}[\rm(i)]
\item If $0 \leq \alpha < N$ and $\beta < -2$ then, there exists $\lambda>0$ so that \eqref{P+} has a non-negative solution in $\R^N$ if and only if
$$
p,\, q \geq \frac{N - \alpha}{N - 2} \quad\text{and}\quad p + q \geq \frac{2N - \alpha}{N - 2}.
$$
\item If $0 \leq \alpha < N$ and $\beta > -1 + \frac{1}{q}$ then, there exists $\lambda>0$ so that \eqref{P+} has a non-negative solution in $\R^N$ if and only if
$$
p,\, q >\frac{N - \alpha}{N - 2} \quad\text{and}\quad p + q > \frac{2N - \alpha}{N - 2}.
$$
\end{enumerate}
\end{cor}
Finally, we leave open the following cases (see Table 1 below) which essentially involve a small interval that contains the exponent $\beta > \alpha - N$.
\begin{table}[h]
\centering
\caption*{Open cases for \eqref{P+}}\vspace{-0.2cm}
\begin{tabular}{cccc}
\toprule\toprule
\vspace{0.1cm}$p$ & $q$ & $p + q$ & $\beta$ \\ 
\midrule
\vspace{0.15cm} $p > \max\left\{1, \frac{N - \alpha}{N - 2}\right\}$ \quad & \quad $q > \frac{N - \alpha}{N - 2}$ \quad & \quad $p + q = \frac{2N - \alpha}{N - 2}$ \quad & \quad $-1\leq \beta\leq -1+\frac{1}{p+q} $\vspace{0.15cm}\\ 
\vspace{0.15cm} $1\leq p = \frac{N - \alpha}{N - 2}$ \quad & \quad $q = \frac{N}{N - 2}$ \quad & \quad $p + q = \frac{2N - \alpha}{N - 2}$ \quad & \quad $-2\leq \beta \leq  -2 + \frac{1}{q}$  \vspace{0.1cm}\\ 
\vspace{0.15cm} $p = \frac{N}{N - 2}$ \quad & \quad $q = \frac{N-\alpha}{N - 2}$ \quad & \quad $p + q = \frac{2N - \alpha}{N - 2}$ \quad & \quad $-2\leq \beta \leq  -2 + \frac{1}{q}$  \vspace{0.1cm}\\ 
\vspace{0.1cm} $p > \frac{N}{N - 2}$ \quad & \quad  $1< q = \frac{N-\alpha}{N - 2}$ \quad & \quad $p + q > \frac{2N - \alpha}{N - 2}$ \quad & \quad $-1\leq \beta \leq  -1+\frac{1}{q}$  \vspace{0.1cm}\\ 
\vspace{0.1cm} $p > \frac{N}{N - 2}$ \quad & \quad $0 < q \leq  1$ \quad & \quad $p + q = \frac{2N - \alpha}{N - 2}$ \quad & \quad $\alpha - N < \beta \leq  -1+\frac{1}{p + q}$  \vspace{0.1cm}\\ 
\vspace{0.1cm} $p > \frac{N}{N - 2}$ \quad & \quad $0 < q \leq  1$ \quad & \quad $p + q > \frac{2N - \alpha}{N - 2}$ \quad & \quad $\beta > \alpha - N$ \vspace{0.1cm}\\ 
\bottomrule\bottomrule
\end{tabular}
\bigskip
\caption{The range of exponents not covered in Theorem \ref{thm2} and Theorem \ref{thm3}. }
\end{table}
\vspace{-0.3cm}

The remainder of this paper is organized as follows. Section 2 contains some preliminary results, including various integral estimates involving the potential $\K1$. Section 3 and Section 4 provide the proofs of Theorem \ref{thm1} and Theorem \ref{thm2} respectively. Section 5 and Section 6 are devoted to the proofs of Theorem \ref{thm3} and \ref{thm4}. 
The summary of exponents $p,q,\alpha,\beta$ for which Theorem \ref{thm2}, Theorem \ref{thm3} and Theorem \ref{thm4} yield the existence, non-existence and open cases for a non-negative solutions of \eqref{P+}, assuming a proper $\lambda > 0$ large, is presented in the Table 2 below. 
\begin{landscape}
\begin{table}[ht]
\caption*{Existence \& Non-existence for \eqref{P+}}\vspace{-0.2cm}
\begin{tabular}{c|c|c|c|c|cc}
\hline\hline
\boldmath ${p \geq 1}$ & \boldmath$q > 0$ & \boldmath$p + q$ & \boldmath$\alpha$ & \boldmath$\beta > \alpha - N$ & {\bf Exist/Non-exist} & {\bf Result}\\ \hline\hline
$p < \frac{N - \alpha}{N - 2}$ & $q > 0$  & $p + q > 0$ & $\alpha \in [0, 2]$ & $\beta > \alpha - N$ & Non-exist & Theorem \ref{thm2}(ii)\\ \hline
\multirow{6}{*}{$p = \frac{N - \alpha}{N - 2}$} & $0 < q < \frac{N}{N - 2}$ & $p + q < \frac{2N - \alpha}{N - 2}$ & $\alpha \in [0, 2]$ & $\beta > \alpha - N$ & Non-exist & Theorem \ref{thm2}(iv) \\ \cline{2-7} 
&\multirow{3}{*}{$q = \frac{N}{N - 2}$} & \multirow{3}{*}{$p + q = \frac{2N - \alpha}{N - 2}$} & \multirow{3}{*}{$\alpha \in [0, 2]$} & $\beta < -2$ & Exist & Theorem \ref{thm3}(v) \\ \cline{5-7} 
&&&& $\beta > -2 + \frac{1}{q}$ & Non-exist & Theorem \ref{thm2}(iii)\&(ix)  \\ \cline{5-7} 
&&&& $-2 \leq \beta \leq -2 + \frac{1}{q}$ & ? & ? \\ \cline{2-7} 
& \multirow{2}{*}{$q > \frac{N}{N - 2}$} &  \multirow{2}{*}{$p + q > \frac{2N - \alpha}{N - 2}$} & \multirow{2}{*}{$\alpha \in [0, 2]$} & $\beta \geq -1$ & Non-exist & Theorem \ref{thm2}(iii) \\ \cline{5-7} 
&& & & $\beta < -1$ & Exist & Theorem \ref{thm3}(ii) \\ \hline
\multirow{5}{*}{$\frac{N - \alpha}{N - 2} < p < \frac{N}{N - 2}$} & $0 < q \leq \frac{N - \alpha}{N - 2}$ & $p + q < \frac{2N - \alpha}{N - 2}$ & $\alpha \in [0, 2]$ & $\beta > \alpha - N$ & Non-exist & Theorem \ref{thm2}(iv) \\ \cline{2-7}
& \multirow{4}{*}{$q > \frac{N - \alpha}{N - 2}$} & \multirow{3}{*}{$p + q = \frac{2N - \alpha}{N - 2}$} & \multirow{3}{*}{$\alpha \in [0, 2]$} & $\beta > -1 + \frac{1}{p + q}$ & Non-exist & Theorem \ref{thm2}(v) \\ \cline{5-7}
& & & & $-1 < \beta \leq -1 + \frac{1}{p + q}$ & ? & ? \\ \cline{5-7}
& & & & $\beta < -1$ & Exist & Theorem \ref{thm3}(iv) \\ \cline{3-7}
& & $p + q > \frac{2N - \alpha}{N - 2}$ & $\alpha \in [0, 2]$ & $\beta > \alpha - N$ & Exist & Theorem \ref{thm3}(i) \\ \hline
\multirow{5}{*}{$p = \frac{N}{N - 2}$} & $0 < q < \frac{N - \alpha}{N - 2}$ & $p + q < \frac{2N - \alpha}{N - 2}$ & $\alpha \in [0, 2]$ & $\beta > \alpha - N$ & Non-exist & Theorem \ref{thm2}(iv) \\ \cline{2-7}
& \multirow{3}{*}{$q = \frac{N - \alpha}{N - 2}$} & \multirow{3}{*}{$p + q = \frac{2N - \alpha}{N - 2}$} & \multirow{2}{*}{$\alpha \in [0, 2]$} &  $\beta > -2 + \frac{1}{q}$ & Non-exist & Theorem \ref{thm2}(viii) \\ \cline{5-7}
& & & &  $-2 \leq \beta \leq -2 + \frac{1}{q}$ & ? & ? \quad \\ \cline{4-7}
& & & $\alpha \in [0, N)$ & $\beta < -2$ & Exist & Theorem \ref{thm3}(vi)  \\ \cline{2-7}
& $q > \frac{N - \alpha}{N - 2}$ & $p + q > \frac{2N - \alpha}{N - 2}$ & $\alpha \in [0, N]$ & $\beta > \alpha - N$ & Exist & Theorem \ref{thm3}(i)  \\ \hline
\multirow{9}{*}{$p > \frac{N}{N - 2}$} & $0 < q < \frac{N - \alpha}{N - 2}$ & $p + q < \frac{2N - \alpha}{N - 2}$ & $\alpha \in [0, N)$ & $\beta > \alpha - N$ & Non-exist & Theorem \ref{thm2}(iv) \\ \cline{2-7}
& \multirow{3}{*}{$0 < q \leq 1$} & $p + q = \frac{2N - \alpha}{N - 2}$ & $\alpha \in [0, N)$ & $\beta > -1 + \frac{1}{p + q}$ & Non-exist & Theorem \ref{thm2}(v) \\ \cline{3-7}
& & $p + q = \frac{2N - \alpha}{N - 2}$ & $\alpha \in [0, N)$ & $\alpha - N < \beta \leq -1 + \frac{1}{p + q}$ & ? & ? \\ \cline{3-7}
& & $p + q > \frac{2N - \alpha}{N - 2}$ & $\alpha \in [0, N)$ & $\beta > \alpha - N$ & ? & ? \\ 
\cline{2-7}
& $1 < q < \frac{N - \alpha}{N - 2}$ &  $p + q >0$ &  $\alpha \in [0, 2)$ & $\beta > \alpha - N$ & Non-exist & Theorem \ref{thm2}(vi) \\ \cline{2-7}
& \multirow{3}{*}{$q = \frac{N - \alpha}{N - 2}$} & \multirow{3}{*}{$p + q > \frac{2N - \alpha}{N - 2}$} & \multirow{3}{*}{$\alpha \in [0, 2)$} & $\beta > -1 + \frac{1}{q}$ & Non-exist & Theorem \ref{thm2}(vii) \\ \cline{5-7}
& & & & $-1 \leq \beta \leq -1 + \frac{1}{q}$ & ? & ? \\ \cline{5-7}
& & & & $\beta < -1$ & Exist & Theorem \ref{thm3}(iii) \\ \cline{2-7}
& $q > \frac{N - \alpha}{N - 2}$ & $p + q > \frac{2N - \alpha}{N - 2}$ & $\alpha \in [0, N]$ & $\beta > \alpha - N$ & Exist & Theorem \ref{thm3}(i)\&\ref{thm4}\\ 
\hline\hline
\end{tabular}
\bigskip
\caption{The range of exponents for the existence, non-existence and open cases of a non-negative solution in Theorem \ref{thm2}, \ref{thm3} and \ref{thm4}. }
\end{table}
\end{landscape}
\section{Preliminary Results}
In this section we collect some auxiliary results for proving Theorem \ref{thm2}, Theorem \ref{thm3} and Corollary \ref{cor}. Our first result in this sense extends the integral estimates obtained in \cite{CZ23, G23,  GKS20, GKS21, MGZ23, WZ23}.
\begin{lemma}\label{lm1} Let $f \in L^1_{loc}(\R^N)$, $f \geq 0$ and let $\K1$ be given by \eqref{ptt}. 
\begin{enumerate}[\rm(i)]
\item There exists a constant $ C > 0$ such that
\begin{equation}\label{l21a}
(\K1 * f)(x) \geq C\abs{x}^{-\alpha}\log^{\beta}(1 + \abs{x}) \quad \text{in } \R^N \setminus B_1.
\end{equation}
Furthermore, if $f(x) \geq c\abs{x}^{-N}$ in $\R^N \setminus B_1$ for some $c >0$ and $\beta \leq 0$, then, there exists a constant $C > 0$ such that
\begin{equation}\label{l21b}
(\K1 * f)(x) \geq C|x|^{-\alpha}\log^{1 + \beta}(1 + \abs{x})\quad \text{in } \R^N \setminus B_1.
\end{equation}
\item Suppose 
$$f(x) \geq c\abs{x}^{-\sigma}\log^{\kappa}(1 + \abs{x})\quad \text{in }\R^N \setminus B_1\quad\text{for some } c >0, \kappa \in \R \text{ and } \sigma \geq 0.$$
Then, for all $x \in \R^N \setminus B_1$ one has:
\[
\begin{split}
(\K1 * f)(x) = \infty \quad &\text{if }\;\;N - \alpha - \sigma > 0,\\
&\text{or }\;N - \alpha - \sigma = 0,\quad1 + \beta + \kappa \geq 0.
\end{split}
\]
Furthermore, there exists a constant $C > 0$, such that for all $x \in \R^N \setminus B_1$ one has:
\[
(\K1 * f)(x) \geq C\left\{
\begin{aligned}
&\abs{x}^{N - \alpha - \sigma}\log^{\beta + \kappa}(1 + \abs{x}) &\text{if }  N - \alpha - \sigma < 0,&\\[0.1in]
&\log^{1 + \beta + \kappa}(1 + \abs{x}) &\text{if }N - \alpha - \sigma = 0,& \quad 1 + \beta + \kappa < 0.
\end{aligned}\right.
\]
\item Suppose 
$$f(x) \leq c(A + \abs{x})^{-\sigma}\log^{\kappa}(A + \abs{x})\quad \text{in }\R^N,$$
for some $c > 0$, $\sigma \geq N - \alpha$, $A > 1$ and $\kappa \in \R$.
Then, there exists $C =C(N, \alpha, \beta, \sigma, \kappa, A) > 0$ such that for all $x\in \R^N$ one has following estimates:\\[0.2cm]
If $\alpha < N$, then
$$ 
(\K1 * f)(x)\leq C\begin{cases}
\log^{1 + \beta + \kappa}(A + |x|)&\text{if }\sigma=N-\alpha,\, 1 + \beta + \kappa < 0,\\[0.2cm]
(A + \abs{x})^{N - \alpha - \sigma}\log^{\beta + \kappa}(A + \abs{x})&\text{if } N - \alpha < \sigma < N,\\[0.2cm]
(A + \abs{x})^{-\alpha}\log^{\beta}(A + \abs{x})&\text{if }\sigma > N\text{ or }\sigma = N,\, \kappa < -1, \\[0.2cm]
(A + \abs{x})^{-\alpha}\log^{1 + \beta + \kappa}(A + \abs{x})&\text{if }\sigma = N,\, \kappa > -1,\\[0.2cm]
(A + |x|)^{-\alpha}\log^{\beta}(A + |x|)\log(\log(e + |x|))&\text{if }\sigma = N,\, \kappa = -1. \\[0.2cm]
\end{cases} 
$$
If $\alpha = N$, then
$$
(\K1 * f)(x) \leq C\begin{cases}
(A + |x|)^{-\sigma}\log^{1 + \beta + \kappa}(A + \abs{x})&\text{if }0 < \sigma < N,\\[0.2cm]
(A + |x|)^{-N}\log^{1 + \beta + \kappa}(A + \abs{x})&\text{if }\sigma = N,\,\kappa > -1\\[0.2cm]
(A + |x|)^{-N}\log^{\beta}(A + |x|)&\text{if }\sigma > N \text{ or } \sigma = N,\,\kappa < -1, \\[0.2cm]
(A + |x|)^{-N}\log^{\beta}(A + |x|)\log(\log(e + |x|))&\text{if }\sigma = N,\, \kappa = -1. \\[0.2cm]
\end{cases}  
$$
\end{enumerate}
\end{lemma}
\begin{proof}
(i) Let $x \in \R^N\setminus B_1$. Then, for all $y \in B_{\frac{1}{2}}$ one has 
$$
\frac{|x|}{2} \leq |x| - |y| \leq |x - y| \leq |x| + |y| \leq \frac{3}{2}|x|.
$$ 
Thus
\[K_{\alpha, \beta}(x - y) \geq c|x|^{-\alpha}\log^{\beta}(1 + |x|)\quad\text{for some } c > 0.\]
It follows that
\begin{align*}
(\K1 * f)(x) &\geq \int_{|y| \leq \frac{1}{2}}K_{\alpha, \beta}(x - y)f(y) dy \nonumber\\
&\geq c|x|^{-\alpha}\log^{\beta}(1 + |x|)\int_{|y| \leq \frac{1}{2}}f(y) dy \nonumber\\
&\geq C|x|^{-\alpha}\log^{\beta}(1 + |x|)\quad \text{for all } x \in \R^N \setminus B_1. 
\end{align*}
Furthermore, let $f(x) \geq c|x|^{-N}$ in $\R^N \setminus B_1$ and $\beta \leq 0$. Then, for all $1 \leq |y| \leq |x|$, we have $|x - y| \leq 2|x|$ and
\[K_{\alpha, \beta}(x - y) \geq c|x|^{-\alpha}\log^{\beta}(1 + |x|)\quad\text{for some } c > 0.\]
Thus, by Coarea formula, we estimate
\begin{align*}
 (\K1 * f)(x) &\geq \int_{|x - y|\leq 2|x|} K_{\alpha, \beta}(x - y) f(y)dy \\
 &\geq C|x|^{-\alpha}\log^{\beta}(1 + |x|)\int_{1 < |y| \leq |x|} |y|^{-N}dy\\
 &= C|x|^{-\alpha}\log^{\beta}(1 + |x|)\int_1^{|x|}t^{-1}dt\\ 
 &\geq C|x|^{-\alpha}\log^{1 + \beta}(1 + |x|)\quad \text{for all } x \in \R^N \setminus B_1. 
\end{align*}
(ii) Let $|y| \geq 2|x| \geq 2$. Then,
\begin{equation}\label{tta}
\frac{|y|}{2} \leq |y| - |x| \leq |x - y| \leq |x| + |y| \leq \frac{3}{2}|y|.
\end{equation}
Thus,
\begin{equation}\label{ttb}
K_{\alpha, \beta}(x - y) \geq c|y|^{-\alpha}\log^{\beta}(1 + |y|)\quad\text{for some } c > 0.
\end{equation}
Using \eqref{ttb} and Coarea formula, for all $|x| > 1$ we estimate
\begin{align*}
   (\K1 * f)(x) &\geq C\int_{|y| \geq 2|x|} |y|^{-\alpha - \sigma}\log^{\beta + \kappa}(1 +|y|)dy\\
   &= C\int_{2|x|}^{\infty}t^{N - \alpha - \sigma - 1}\log^{\beta + \kappa}(1 + t)dt\\
   &\geq \left\{\begin{aligned}
       &\infty \quad &\text{if } \sigma < N - \alpha&,\\
       &\infty \quad &\text{if } \sigma = N - \alpha&,\; 1 + \beta + \kappa \geq 0,\\
       &C|x|^{N - \alpha - \sigma}\log^{\beta + \kappa}(1 + |x|) \quad &\text{if } \sigma > N - \alpha&,\\
       &C\log^{1 + \beta + \kappa}(1 + |x|) \quad &\text{if } \sigma = N - \alpha&,\; 1 + \beta + \kappa < 0.
   \end{aligned}\right.
\end{align*}
(iii) Let $x \in \R^N$. We split the convolution integral as follows:
\begin{equation}\label{lm}
    \begin{split}
      (\K1 * f)(x) &= I_1 + I_2 + I_3\\
      &= \left\{ \int_{|y| \geq 2|x|} + \int_{\frac{|x|}{2} \leq |y| < 2|x|} + \int_{|y| < \frac{|x|}{2}} \right\} \K1(x - y)f(y) dy. 
    \end{split}
\end{equation}
To estimate each component of \eqref{lm}, we divide our analysis into two cases.\\
\textbf{Case 1:} $|x| > 1$. \\
For $|y| \geq 2|x|$, we use \eqref{tta} and we have 
\begin{equation}\label{ssa}
\frac{|y|}{2} \leq |x - y| \leq \frac{3|y|}{2}.
\end{equation} 
Thus,
\begin{equation}\label{ssb}
K_{\alpha, \beta}(x - y) \leq c|y|^{-\alpha}\log^{\beta}(1 + |y|)\quad\text{for some } c > 0.
\end{equation}
Hence,
\begin{equation}\label{I1a}
I_1 \leq C \int_{|y| \geq 2|x|} |y|^{-\alpha-\sigma}\log^{\beta + \kappa}(A + |y|)dy = C \int_{2|x|}^{\infty}t^{N - \alpha - \sigma - 1}\log^{\beta + \kappa}(A + t) dt.
\end{equation}
If $\sigma = N - \alpha$ and $1 + \beta + \kappa < 0$, we use the fact that $t > \frac{A + t}{A + 1}$ whenever $t > 1$. Then, from \eqref{I1a} one has
\begin{align}
    I_1 &\leq C\int_{2|x|}^{\infty}t^{-1}\log^{\beta + \kappa}(A + t) dt \nonumber \\
    &\leq C(A + 1)\int_{2|x|}^{\infty}\frac{\log^{\beta + \kappa}(A + t)}{A + t}dt \nonumber \\
    &\leq C\log^{1 + \beta + \kappa}(A + |x|)\label{I1b}.      
\end{align}
If $\sigma > N - \alpha$, then, from \eqref{I1a} and the inequality
\begin{equation}\label{ax}
|x|\geq \frac{A+|x|}{A+1}\quad \mbox{ for all }|x|>1,
\end{equation}
we obtain
\begin{equation}\label{I1c}
\begin{aligned}
I_1 &\leq  C|x|^{N - \alpha - \sigma}\log^{\beta + \kappa}(A + |x|) \\[0.2cm]
& \leq C\Big( \frac{A+|x|}{A+1}\Big)^{N - \alpha - \sigma}\log^{\beta + \kappa}(A + |x|)\\[0.2cm]
& \leq C(A+|x|)^{N - \alpha - \sigma}\log^{\beta + \kappa}(A + |x|).
\end{aligned}
\end{equation}
Combine \eqref{I1b} and \eqref{I1c} to conclude
\begin{equation}\label{I1}
    I_1 \leq C\left\{
    \begin{aligned}
    &\log^{1 + \beta + \kappa}(A + |x|)\quad&\text{if } N - \alpha - \sigma = 0&, 1 + \beta + \kappa < 0,\\
    &(A + |x|)^{N - \alpha - \sigma}\log^{\beta + \kappa}(A + |x|)\quad&\text{if } N - \alpha - \sigma < 0&.
    \end{aligned}\right.
\end{equation}
To estimate $I_2$, we have $\frac{|x|}{2} \leq |y| < 2|x|$, which implies $|x - y| < 3|x|$. It follows that
\begin{align}\label{I2a}
    I_2 &\leq C(A + |x|)^{-\sigma}\log^{\kappa}(A + |x|) \int_{|x - y| < 3|x|}|x - y|^{-\alpha}\log^{\beta}(1 + |x - y|)dy \nonumber\\[0.2cm]
    &\leq C(A + |x|)^{-\sigma}\log^{\kappa}(A + |x|)\int_0^{3|x|}t^{N - \alpha - 1}\log^{\beta}(1+ t)dt.
\end{align}
If $\alpha < N$, then from \eqref{I2a} one has
\begin{align}\label{I2b}
    I_2 &\leq C(A + |x|)^{-\sigma}\log^{\kappa}(A + |x|)|x|^{N - \alpha}\log^{\beta}(1+ |x|) \nonumber\\
    &\leq C(A + |x|)^{N - \alpha - \sigma}\log^{\beta + \kappa}(A + |x|).
\end{align}
If $\alpha = N$, then \eqref{ptt} yields $\beta > 0$ and from \eqref{I2a} together with $\log(1+t)\leq t$ we deduce
\begin{align}\label{I2}
I_2 &\leq C(A + |x|)^{-\sigma}\log^{\kappa}(A + |x|)\int_0^{3|x|}t^{-1}\log^{\beta}(1 + t)dt \nonumber\\
&\leq C(A + |x|)^{-\sigma}\log^{\kappa}(A + |x|)\left\{\int_0^{1}t^{\beta -1}dt + 2\int_{1}^{3|x|}\frac{\log^{\beta}(1 + t)}{1 + t}dt\right\} \nonumber\\
&\leq C(A + |x|)^{-\sigma}\log^{1 + \beta + \kappa}(A + |x|).
\end{align}
Thus, \eqref{I2b} and \eqref{I2} we have
\begin{equation}\label{cas2}
I_2\leq C
\begin{cases}
(A + |x|)^{N - \alpha - \sigma}\log^{\beta + \kappa}(A + |x|) &\mbox{ if }\alpha<N,\\[0.2cm]
(A + |x|)^{-\sigma}\log^{1 + \beta + \kappa}(A + |x|) &\mbox{ if }\alpha=N.
\end{cases}
\end{equation}
Finally, to estimate $I_3$, for $|y| < \frac{|x|}{2}$, one has $ \frac{|x|}{2} < |x - y| <  \frac{3|x|}{2}$. Thus,
$$
K_{\alpha, \beta}(x - y) \leq c|x|^{-\alpha}\log^{\beta}(1 + |x|)\quad\text{for some } c > 0.
$$
It follows that
\[I_3 \leq C|x|^{-\alpha}\log^{\beta}(1 + |x|) \int_{|y| < \frac{|x|}{2}} (A + |y|)^{-\sigma}\log^{\kappa}(A + |y|)dy. \]
Using \eqref{ax} again, we further estimate:
\begin{align}\label{I3}
    I_3 &\leq C(A + |x|)^{-\alpha}\log^{\beta}(A + |x|)\int_0^{\frac{|x|}{2}}(A + t)^{N - \sigma - 1}\log^{\kappa}(A + t) dt \nonumber\\ 
    &\leq C(A + |x|)^{-\alpha}\log^{\beta}(A + |x|)\left\{1 + \int_{\frac{1}{2}}^{\frac{|x|}{2}}(A + t)^{N - \sigma - 1}\log^{\kappa}(A + t)dt\right\} \nonumber\\
    &\leq C\left\{
  \begin{aligned}
     &(A + |x|)^{N - \alpha - \sigma}\log^{\beta + \kappa}(A + |x|)&\text{if }\sigma < N&,\\
     &(A + |x|)^{-\alpha}\log^{\beta}(A + |x|)&\text{if }\sigma > N&\text{ or }\sigma = N,\, \kappa < -1,\\
     &(A + |x|)^{-\alpha}\log^{1 + \beta + \kappa}(A + |x|)&\text{if }\sigma = N&,\, \kappa > -1,\\
     &(A + |x|)^{-\alpha}\log^{\beta}(A + |x|)\log(\log(e + |x|))&\text{if }\sigma = N&,\, \kappa = -1.\\
  \end{aligned}\right.
\end{align}
\textbf{Case 2: }$|x| \leq 1$.\\
For $I_1$, we see that \eqref{ssa} and \eqref{ssb} still hold and then 
\begin{align*}
I_1 &\leq \int_{|y| \geq 2|x|}|y|^{-\alpha}\log^{\beta}(1 + |y|)(A + |y|)^{-\sigma}\log^{\kappa}(A + |y|)dy \nonumber\\
&\leq C\int_{0}^{\infty}t^{N - \alpha - 1}(A + t)^{-\sigma}\log^{\beta}(1 + t)\log^{\kappa}(A + t)dt \nonumber\\
\end{align*}
Using the fact that $ct \leq \log(1 + t) \leq t$ for all $0 < t \leq 1$, if either 
\begin{itemize}
\item $\sigma>N-\alpha$;
\item or $\sigma=N-\alpha$ and $\beta+\kappa<0$,
\end{itemize} 
we conclude
$$
\begin{aligned}
I_1 & \leq C\int_{0}^{1}t^{N - \alpha + \beta - 1}(A + t)^{- \sigma}\log^{\kappa}(A + t)dt + C\int_1^{\infty}t^{N - \alpha - \sigma - 1}\log^{\beta + \kappa}(1 + t)dt \\[0.2cm]
& \leq C\int_{0}^{1}t^{N - \alpha + \beta - 1} dt + C\int_1^{\infty}t^{N - \alpha - \sigma - 1}\log^{\beta + \kappa}(1 + t)dt\\[0.2cm]
&\leq C.
\end{aligned}
$$
For $I_2$, we observe that the estimate \eqref{I2b} remains valid for the case $\alpha < N$ whenever $|x| \leq 1$. For the case $\alpha = N$, we note that \eqref{ptt} yields $\beta>0$ and we deduce
\begin{equation*}
    \begin{split}
        I_2 &\leq C(A + |x|)^{-\sigma}\log^{\kappa}(A + |x|)\int_0^{3}t^{-1}\log^{\beta}(1 + t)dt\\
        &\leq C(A + |x|)^{-\sigma}\log^{\kappa}(A + |x|)\int_0^{3}t^{\beta-1}dt \leq C.
    \end{split}
\end{equation*}
For $I_3$, from $|y| < \frac{|x|}{2} \leq \frac{1}{2}$ one has $|x - y| < \frac{3}{2}$. It follows that
\begin{equation*}
\begin{split}
     I_3 &\leq \int_{|x - y| < \frac{3}{2}}|x - y|^{-\alpha}\log^{\beta}(1 + |x - y|)(A + |y|)^{-\sigma}\log^{\kappa}(A + |y|)dy\\
     &\leq C\int_{|x - y| < \frac{3}{2}}|x - y|^{-\alpha}\log^{\beta}(1 + |x - y|)dy\\
     &\leq C\int_0^{\frac{3}{2}}t^{N - \alpha - 1}\log^{\beta}(1 + t)dt\\
     &\leq C\int_0^{\frac{3}{2}}t^{N - \alpha + \beta - 1}dt \leq C.
\end{split}
\end{equation*}
Thus, we observe that $I_1$, $I_2$ and $I_3$ are bounded by a positive constant.  Hence, \eqref{I1}, \eqref{cas2} and \eqref{I3} hold for $|x|\leq 1$ as well with a larger constant $C>0$. Therefore, combining \eqref{lm}, \eqref{I1}, \eqref{cas2} and \eqref{I3}, we reach the conclusion for all $x \in \R^N$.
\end{proof}
We finish the presentation of this Section with the following useful result. 
\begin{lemma}\label{lm2}
\begin{enumerate}
\item[{\rm (i)}] There are no positive, non-constant superharmonic functions in $\R^N$, $N=1,2$.
\item[{\rm (ii)}] Any positive superharmonic function $u \in  \mathcal{C}^2(\R^N)$, $N\geq 3$, satisfies
\begin{equation}\label{ssup}
u(x) \geq C|x|^{2 - N}\quad \text{in } \R^N\setminus B_1,
\end{equation}
for some positive constant $C>0$.
\end{enumerate}
\end{lemma}
If $N=1$, the proof of Lemma \ref{lm2}(i) is obvious. If $N=2$, Lemma \ref{lm2}(i) is given in \cite[Theorem 29, page 30]{PW84}. Finally, if $N\geq 3$ the proof of Lemma \ref{lm2}(ii) follows from a standard maximum principle argument. It is also a consequence of \cite[Lemma 2.3]{SZ02}.
\section{Proof of Theorem \ref{thm1}}
Assume $u \in \mathcal{C}^4(\R^N)$, $u \geq 0$ and $u \not\equiv 0$ is a solution of \eqref{P-}. We claim that $\Delta u \geq 0$ in $\R^N$.  Letting $v = \Delta u$, then \eqref{P-} implies
\begin{equation}\label{dv1}
-\Delta v + \lambda v \geq 0 \quad \text{in }\;\;\R^N.
\end{equation}
Assume by contradiction that $v(x_0) < 0$ for some $x_0 \in \R^N$. Replacing $v(x)$ by $v(x + x_0)$, we may assume that $v(0) < 0$. Denote by $\overline{v}(r)$ the spherical average of $v(x)$ over $\partial B_r$, so that 
\[\overline{v}(r) = \fint_{\partial B_r} v(x) d\sigma(x)\quad\mbox{ for all }r\geq 0.\]
It follows from \eqref{dv1} that
\begin{equation}\label{dv2}
    -\Delta \overline{v}(r) +  \lambda\overline{v}(r) \geq 0 \quad \text{for all } r \geq 0.
\end{equation}
Since $\overline{v}(0)=v(0)<0$, we may define
$$
\rho:= \sup\{r > 0: \overline{v}(t) < 0\quad \text{for all }\; 0 < t < r\}>0.
$$
We claim that $\rho = \infty$. Assuming the contrary, we have $\rho<\infty$ and by continuity argument it follows $\overline{v}(\rho) = 0$. Then, \eqref{dv2} implies 
$$
\Delta \overline{v}(r) \leq  \lambda \overline{v}(r) < 0\quad\mbox{for all }0\leq r<\rho.
$$
This means that $(r^{N -1}\overline{v}')' < 0$ for all $0\leq r<\rho$. Thus, the mapping 
$[0, \rho]\ni r \longmapsto r^{N -1}\overline{v}'(r)$ is decreasing which yields  $\overline{v}' < 0$ in $[0, \rho]$.  In particular, this implies $\overline{v}(\rho) < \overline{v}(0) < 0$  which contradicts $\overline v(\rho)=0$. This proves $\rho=\infty$. Hence,
$$
\overline v(r)=\Delta \overline u(r)\leq \overline v(0)=-c\quad\mbox{for all }r\geq 0,\quad \text{where } c > 0 \text{ is a constant.}
$$
This yields
$$
\big(r^{N-1}\overline u'(r)\big)'\leq -cr^{N-1}\quad \mbox{for all }r\geq 0.
$$
Integrating twice in the above estimate we obtain
 $$
 \overline{u}(r) < -\frac{cr^2}{2N} + \overline{u}(0) \quad\mbox{for all }r\geq 0.
$$
Therefore, $\overline{u}(r) < 0$ for $r$ large, which contradicts our assumption that $u(x) \geq 0$ for any $x \in \R^N$. We conclude that $\Delta u \geq 0$ in $\R^N$. 
\\
(i) Take $r_0 > 0$ so that $\sup_{B_{r_0}} u > 0$. If $p > 1$, then for $R > 2r_0$, by the weak Harnack Inequality (see \cite[Theorem 1.3]{Tru67}), we have
$$
\left(\fint_{B_{R}} u^p\right)^{\frac{1}{p}} \geq c\sup\limits_{B_{r_0}} u \geq C,\quad \text{where } c, C > 0 \text{ are constants}.
$$
It follows that
$$
\int_{B_R} u^p \geq CR^N.
$$
If $p = 1$, by using $\Delta u \geq 0$ in $\R^N$ and taking the average, we have $(r^{N - 1}\overline{u}')' \geq 0$ for all $r \geq 0$. This means that $\overline{u}' \geq 0$ and $\overline{u}$ is increasing. Thus, $\overline{u}(r) \geq \overline{u}(r_0) = c > 0$ for all $r \geq r_0$ and $c > 0$ is a constant independent of $r$. Then, for any $R > 2r_0$, we have
\[\begin{split}
\int_{B_R} u &\geq \int_{B_R\setminus B_{r_0}} u = \sigma_N \int_{r_0}^R r^{N-1} \overline{u}(r) dr\\
&\geq C\int_{r_0}^R r^{N-1} dr = c(R^N - r_0^N) \geq CR^N,
\end{split}\]
where $C$ is a positive constant independent of $R$. Hence, for all $p \geq 1$, we conclude
\begin{equation}\label{harnack}
    \int_{B_R} u^p \geq CR^N.
\end{equation}
Then, using \eqref{harnack}, we estimate
\begin{equation*}
\begin{split}
(\K1 * u^p)(0) &= \int_{\R^N} \K1(y) u^p(y) dy \\
&\geq \int_{\Bo}\frac{u^p(y)\log^{\beta}(1 + \abs{y})}{\abs{y}^{\alpha}}dy \\
&\geq \frac{\log^{\beta}(1 + R)}{R^{\alpha}}  \int_{\Bo}  u^p(y) dy\\
&\geq CR^{N - \alpha}\log^{\beta}(1 + R) \to \infty \quad \text{as } R \to \infty.\\
\end{split}
\end{equation*}
Hence, $(\K1 * u^p)(0) = +\infty$ which contradicts condition \eqref{KKK}.  This shows that \eqref{P-}  has no solutions whenever $p \geq 1$.\\
(ii) Assume that $p < 1$ and $u(x)$ is a non-negative bounded solution of \eqref{P-} such that $u \leq \sup_{\R^N}u = M < \infty$. By applying the weak Harnack inequality as before for the exponent $p + 1$, we obtain
\[\int_{B_R} u^p \geq \frac{1}{M}\int_{B_R}u^{p+1} \geq CR^N.\]
for some positive constant $C$ independent of $R$. It follows that
\begin{equation*}
\begin{split}
(\K1 * u^p)(0) &= \int_{\R^N} \K1(y) u^p(y) dy \\
&\geq \int_{\Bo}\frac{u^p(y)\log^{\beta}(1 + \abs{y})}{\abs{y}^{\alpha}}dy \\
&\geq CR^{N - \alpha}\log^{\beta}(1 + R) \to \infty \quad \text{as } R \to \infty.
\end{split}
\end{equation*}
Thus, \eqref{P-} has no bounded solutions whenever $p < 1$.\\
(iii) Assume $u(x) \geq 0$ is a radial solution of \eqref{P-} and $u \not\equiv 0$. From the previous section, we know that $\Delta u \geq 0$ in $\R^N$. It follows that 
\[r^{N - 1}\Delta u = (r^{N-1}u')' \geq 0 \quad \text{for all } r \geq 0.\]
Then $u' \geq 0$ and $u$ is an increasing function. Thus $u(x) \geq u(x_0) = c > 0$ whenever $|x| \geq |x_0| = r$, for some $r > 0$. Moreover,
\begin{equation*}
\begin{split}
(\K1 * u^p)(0) &= \int_{\R^N} \K1(y) u^p(y) dy \\
&\geq c^p \int_{\R^N \setminus B_r} |y|^{-\alpha}\log^{\beta}(1 + |y|)dy\\
&= c^p \int_r^{\infty} t^{N - \alpha - 1}\log^{\beta}(1 + t)dt = \infty \quad \text{since } \alpha, \beta \text{ satisfy}\; \eqref{ptt}.
\end{split}
\end{equation*}
Hence, $(\K1 * u^p)(0) = +\infty$ which contradicts condition \eqref{KKK}.  This shows that \eqref{P-}  has no radial solutions and concludes our proof.
\qed
\section{Proof of Theorem \ref{thm2}}
Assume $u \in \mathcal{C}^4(\R^N)$, $u \geq 0$ and $u \not\equiv 0$ is a solution of \eqref{P+}. We claim that $-\Delta u \geq 0$ in $\R^N$. Letting $w = -\Delta u$, from  \eqref{P+} we deduce
\begin{equation}\label{dw1}
-\Delta w + \lambda w \geq 0  \quad\text{ in }\;\;\R^N.
\end{equation}
Assume by contradiction that $w(x_0) < 0$ for some $x_0 \in \R$. Replacing $w(x)$ by $w(x + x_0)$, we may assume that $w(0) < 0$. Again, denote by $\overline{w}(r)$ the spherical average of $w(x)$ over $\partial B_r(0)$. It follows from \eqref{dw1} that 
\begin{equation}\label{dw2}
    -\Delta \overline{w}(r) + \lambda \overline{w}(r) \geq 0 \quad \mbox{for all }r\geq 0.
\end{equation}
We next apply a similar argument as in Theorem \ref{thm1} to obtain $\overline{w} \leq  \overline{w}(0) < 0$. Therefore, 
$$
-\Delta \overline{u}(r) = \overline{w}(r) \leq \overline{w}(0) = -c \quad \mbox{for all }r\geq 0,
$$
where $c>0$ is a constant. 
Integrating twice in the above inequality we obtain
$$
 \overline{u}(r) \geq \frac{cr^2}{2N} + \overline{u}(0) \quad \mbox{for all }r\geq 0.
$$
By Coarea formula, we get
\begin{equation*}
\begin{split}
 (\K1 * u^p)(0) &= \int_0^{\infty}\int_{\partial B_r}\frac{u^p(y)\log^{\beta}(1 + \abs{y})}{\abs{y}^{\alpha}}d\sigma(y)dr\\
&= C \int_0^{\infty} r^{N - \alpha - 1}\log^{\beta}(1 + r)\fint_{\partial B_r} u^p(y) d\sigma(y) dr\\
\end{split} 
\end{equation*}
We further apply Jensen's inequality to $u^p(y)$ where $p \geq 1$ to obtain:
\begin{equation*}
\begin{split}
(\K1 * u^p)(0) &\geq C\int_0^{\infty} r^{N - \alpha - 1}\log^{\beta}(1 + r)\left(\fint_{\partial B_r} u(y) d\sigma(y)\right)^p dr\\
&> C\int_0^{\infty} r^{N - \alpha - 1 + 2p}\log^{\beta}(1 + r) = \infty,
\end{split} 
\end{equation*}
which contradicts the condition \eqref{KKK}. Thus, we conclude that $-\Delta u \geq 0$ in $\R^N$. \\
(i) Since $-\Delta u \geq 0$, $u\geq 0$ and $u\neq 0$  in $\R^N$, by the maximum principle we have $u>0$.  Since $1 \leq N \leq 2$, by Lemma \ref{lm2}(i) we deduce $u$ is constant. This is a clear contradiction, since constant functions $u$ do not satisfy \eqref{KKK}. Thus, \eqref{P+} has no non-negative solutions if $N = 1, 2$.\\
(ii) Assume $N \geq 3$. From $-\Delta u \geq 0$ in $\R^N$ and $u\not\equiv 0$, we deduce $u>0$ in $\R^N$. Thus, by \eqref{ssup} in Lemma \ref{lm2}(ii), we deduce $u(x) \geq c\abs{x}^{2 - N}$ in $\R^N\setminus B_1$, for some $c > 0$. It follows that
\begin{align}\label{lm2ii}
    (\K1 * u^p)(0) &\geq \int_{\R^N\setminus B_1}\K1(y)u^p(y) dy \nonumber\\
    &\geq C\int_{\R^N\setminus B_1} \abs{y}^{(2-N)p - \alpha} \log^{\beta}(1 + \abs{y}) dy \nonumber\\
    &\geq C\int_{1}^{\infty} r^{N - 1 - \alpha - p(N - 2)}\log^{\beta}(1 + r) dr.
\end{align}
Hence, $(\K1 * u^p)(0) = +\infty$ if $N - \alpha - (N - 2)p > 0$ which contradicts condition \eqref{KKK}. This shows that \eqref{P+} has no solutions whenever $1 \leq p < \frac{N - \alpha}{N - 2}$. \\
\\
(iii) Assume $N \geq 3$ and $0 \leq \alpha \leq 2$. Let $p = \frac{N - \alpha}{N - 2}$. From \eqref{lm2ii}, we have
\begin{equation*}
\begin{split}
(\K1 * u^p)(0) &\geq \int_{\R^N\setminus B_1}\K1(y)u^p(y) dy \\
&\geq C\int_{1}^{\infty} r^{N - 1 - \alpha - p(N - 2)}\log^{\beta}(1 + r) dr \\
&=  C\int_{1}^{\infty}r^{-1}\log^{\beta}(1 + r) dr = \infty \quad \text{if } \beta \geq -1.
\end{split}
\end{equation*}
Thus, \eqref{P+} has no solutions whenever $p = \frac{N - \alpha}{N - 2}$ and $\beta \geq -1$. \\
\\
(iv) Let $\psi \in C_c^{\infty}(\R^N)$, $0 \leq \psi \leq 1$, supp$(\psi) = \overline{B}_2 \subset \R^N$ and $\psi \equiv 1$ on $\overline{B}_1$. For $R \geq 1$, let $\varphi(x) = \psi^4\left(\frac{x}{R}\right)$. We take $\varphi^2$ as the test function in \eqref{P+}. Observe that
\begin{equation*}
\Delta \varphi^2(x) = \Delta\left(\psi^8\xR \right) 
= \frac{8}{R^2} \psi^6\xR \left\{7|\nabla \psi|^2\xR + \psi\xR \Delta\psi \xR\right\}.
\end{equation*}
Then,
\begin{equation*}
|\Delta \varphi^2| \leq \frac{C}{R^2} \psi^6 \leq  \frac{C}{R^2} \varphi \quad\text{in } \R^N.
\end{equation*}
Similarly, we find
\begin{equation*}
|\Delta^2 \varphi^2| \leq \frac{C}{R^4} \varphi \quad\text{in } \R^N.
\end{equation*}
Thus,
\begin{equation}\label{test1}
|\Delta^2 (\varphi^2) - \lambda\Delta \varphi^2| \leq \frac{C}{R^2} \varphi \quad\text{in } \R^N,
\end{equation}
where $C > 0$ is a positive constant independent of $R$. Using $\varphi^2$ as a test function in \eqref{P+}, we have
\begin{align}\label{test2}
\int_{\R^N}(\K1 * u^p)u^q\varphi^2 &\leq \int_{B_{2R}}\varphi^2(\Delta^2 u - \lambda\Delta u)\nonumber \\
&= \int_{B_{2R}}u (\Delta^2 \varphi^2 - \lambda \Delta \varphi^2) \nonumber\\
&\leq \int_{B_{2R}}u |\Delta^2 \varphi^2 - \lambda \Delta \varphi^2| \nonumber\\
&\leq \frac{C}{R^2}\int_{B_{2R}}u\varphi.
\end{align}
Let $x \in \text{supp}(\varphi^2) = B_{2R}$. Then, for $R>1$ large we have
\begin{align}\label{test3}
(\K1 * u^p)(x) &\geq \int_{B_{2R}} u^p(y) |x - y|^{-\alpha}\log^{\beta}(1 + |x - y|)\varphi dy \nonumber\\       
&\geq (4R)^{-\alpha}\log^{\beta}(1 + 4R)\int_{B_{2R}}u^p\varphi.
\end{align}
From \eqref{test2} and \eqref{test3}, we deduce
\begin{equation}\label{uph1}
    \int_{B_{2R}}u\varphi \geq CR^{2-\alpha}\log^{\beta}(1 + 4R)\left(\int_{B_{2R}}u^p\varphi \right)\left(\int_{B_{2R}}u^q\varphi^2\right).
\end{equation}
We claim next that 
\begin{equation}\label{uph2}
\int_{B_{2R}}u^p\varphi \geq CR^{-N(p-1)} \left(\int_{B_{2R}}u\varphi \right)^{p},
\end{equation}
for some $C>0$. Indeed, if $p=1$, then \eqref{uph2} clearly holds. Assume next $p>1$ and let $p'>1$ be its H\"older's conjugate. We apply the H\"older's inequality to obtain
\begin{equation*}
\int_{B_{2R}}u\varphi \leq \left(\int_{B_{2R}}u^p\varphi \right)^{\frac{1}{p}}\left(\int_{B_{2R}}\varphi \right)^{\frac{1}{p'}} \leq CR^{\frac{N}{p'}} \left(\int_{B_{2R}}u^p\varphi \right)^{\frac{1}{p}}.    
\end{equation*}
This yields \eqref{uph2} which now combined with \eqref{uph1} gives
\begin{equation}\label{uph3}
C' \geq R^{2 - \alpha - N(p - 1)}\log^{\beta}(1 + 4R) \left(\int_{B_{2R}}u\varphi \right)^{p-1}\left(\int_{B_{2R}}u^q\varphi^2\right),
\end{equation}
where $C' > 0$ is a constant. Since $-\Delta u \geq 0$ in $\R^N$, we have $u(x) \geq c|x|^{2 - N}$ in $\R^N \setminus B_1$. Using this estimate in \eqref{uph3}, for $R > 1$ large, we get
\begin{equation}\label{est1}
\begin{split}
C' &\geq CR^{2 - \alpha - N(p-1) + 2(p-1) +N - q(N-2)}\log^{\beta}(1 + 4R)\\
&= CR^{2N - \alpha - (N-2)(p+q)}\log^{\beta}(1 + 4R),
\end{split}    
\end{equation}
where $C$ is a positive constant independent of $R$. Thus, the above estimate yields a contradiction if the exponent of $R$ is positive. Therefore, \eqref{P+} has no solutions if $p + q < \frac{2N - \alpha}{N - 2}$ and $\beta > \alpha - N$.\\
\\
(v) Assume $p + q = \frac{2N - \alpha}{N - 2}$. If $\beta > 0$, the estimate \eqref{est1} still yields a contradiction since \eqref{est1} is equivalent to $C' \geq C\log^{\beta}(1 + 4R)$ for $\beta > 0$ and $R > 1$ large.\\ 
It remains to raise a contradiction in the case where $ \frac{1}{p + q} - 1 < \beta \leq 0$. We also note that we may assume $p>\frac{N-\alpha}{N-2}$, since, otherwise, if $p\leq \frac{N-\alpha}{N-2}$ and $\beta\geq -1$ we know from part (ii) and (iii) above that \eqref{P+} has no non-negative solutions. Thus, we may assume $p>\frac{N-\alpha}{N-2}$. Recall that by \eqref{ssup} in Lemma \ref{lm2}(ii) we have $u(x) \geq c|x|^{2 - N}$ in $\R^N \setminus B_1$, where $c > 0$. Thus, for $|x| > 1$ we have
\begin{equation*}
\begin{split}
(\K1 * u^p)(x) &\geq \int_{|y| \geq |x|} u^p(y) |x-y|^{-\alpha}\log^{\beta}(1 + |x-y|) dy\\
&\geq C\int_{|y| \geq |x|} u^p(y) (2|y|)^{-\alpha}\log^{\beta}(1 + 2|y|) dy\\
&\geq C\int_{|y| \geq |x|} |y|^{-\alpha - p(N -2)}\log^{\beta}(1 + 2|y|) dy\\
&\geq C|x|^{N - \alpha - p(N - 2)}\log^{\beta}(1 + 2|x|).
\end{split}    
\end{equation*}
Combine the above estimate into \eqref{P+} to obtain
\begin{equation*}
\begin{split}
\Delta^2 u - \lambda \Delta u &\geq (\K1 * u^p)u^q\\
&\geq C|x|^{N - \alpha - (p + q)(N - 2)}\log^{\beta}(1 + |x|)\\
&\geq C|x|^{-N}\log^{\beta}|x|\quad\text{if } |x| > 2 \text{ large}.
\end{split}    
\end{equation*}
Thus, $\Delta^2 u - \lambda \Delta u \geq C|x|^{-N}\log^{\beta}|x|$  in $\R^N \setminus B_2$. Let $v = -\Delta u \geq 0$. Then
\[-\Delta v + \lambda v \geq C|x|^{-N}\log^{\beta}|x|\quad \text{in } \R^N\setminus B_2.\]
Let $z(x) = C|x|^{-N}\log^{\beta}|x|$ which satisfies
\[-\Delta z + \lambda z \leq C|x|^{-N}\log^{\beta}|x| \quad \text{if } |x| > r_0 > 2 \text{ large}.\]
Let $c > 0$ be such that $v \geq cz$ on $\partial B_{r_0}$ and $\varepsilon > 0$. Then, there exists $R_{\varepsilon} > r_0 > 2$ so that
\[v + \varepsilon \geq cz\quad \text{in } \R^N \setminus B_{R_{\varepsilon}}.\]
By the maximum principle, we have $v + \varepsilon \geq cz$ in $B_{R_{\varepsilon}}\setminus B_{r_0}$. It follows that
\begin{equation}\label{est2}
v + \varepsilon \geq cz \quad \text{in } \R^N \setminus B_{r_0}.
\end{equation}
Since $\varepsilon > 0$ was arbitrary, \eqref{est2} yields $v \geq cz$ in $\R^N \setminus B_{r_0}$, so that for some constant $C>0$ we have
\[-\Delta u \geq C|x|^{-N}\log^{\beta}|x| \quad \text{in } \R^N \setminus B_{r_0}.\]
Let now $w(x) = |x|^{2-N}\log^{\theta}|x|$, where 
$$
\frac{-\beta}{p + q - 1} < \theta < 1 + \beta.
$$ 
Then, by direct computation one has
\[\begin{split}
-\Delta w &= \theta(N - 2)|x|^{-N}\log^{\theta-1}|x| + \theta(\theta - 1)|x|^{-N}\log^{\theta - 2}|x|\\
&\leq C|x|^{-N}\log^{\beta}|x| \leq -\Delta u \quad \text{in } \R^N \setminus B_{r_0}.    
\end{split}
\]
As above, let $c > 0$ so that $u \geq cw$ on $\partial B_{r_0}$. Take $\varepsilon > 0$. Then, there exists $R_{\varepsilon} > r_0 > 0$ so that $u + \varepsilon \geq cw$ in $\R^N \setminus B_{R_\varepsilon}$. By the maximum principle, we have $u + \varepsilon \geq cw$ in $ B_{R_\varepsilon} \setminus B_{r_0}$, thus $u + \varepsilon \geq cw$ in $\R^N \setminus B_{r_0}$. Since $\varepsilon > 0$ was arbitrary, it follows that
\begin{equation}\label{est3}
u \geq cw \geq c|x|^{2 - N} \log^{\theta}|x| \quad \text{in } \R^N \setminus B_{r_0}.   
\end{equation}
Finally, we use \eqref{est3} into the estimate \eqref{uph3} to obtain
\begin{equation}\label{thm2c11}
\begin{split}
C^{'} &\geq R^{2 - \alpha - N(p - 1)}\log^{\beta}R \left(\int_{B_{2R}\setminus B_{r_0}}u \right)^{p-1}\left(\int_{B_{2R}\setminus B_{r_0}}u^q \right)\\
&\geq CR^{2 - \alpha - N(p - 1) + 2(p - 1) + N - q(N - 2)}\log^{\beta + (p - 1 + q)\theta}R\\
&= C\log^{\beta + (p + q - 1)\theta}R \to \infty \quad \text{as } R \to \infty.
\end{split}
\end{equation}
Hence, \eqref{P+} has no non-negative solutions if $p + q = \frac{2N - \alpha}{N - 2}$ and $\beta > \frac{1}{p + q} - 1$.\\
\\
(vi-vii) Assume $N \geq 3$, $q>1$ and $0 \leq \alpha < 2$. We fix $\delta \in (1, q)$ and let $k\geq 1$ be such that 
\begin{equation}\label{eee0}
k > \frac{4}{\delta - 1}.
\end{equation} 
Let $\psi \in C_c^{\infty}(\R^N)$, $0 \leq \psi \leq 1$, supp$(\psi) = \overline{B}_2 \subset \R^N$ and $\psi \equiv 1$ on $\overline{B}_1$. For $R \geq 1$, let $\varphi(x) = \psi^k\left(\frac{x}{R}\right)$.  
We use $\varphi^\delta$ as a test function in \eqref{P+} and obtain 
\begin{equation}\label{tf}
\int_{B_{2R}}  (\Delta^2 u - \lambda \Delta u) \varphi^\delta \geq \int_{B_{2R}}(\K1 * u^p)u^q\varphi^\delta.
\end{equation}
We next estimate from above the left-hand side integral in \eqref{tf}. 
Applying a similar argument to that in part (iv) where we deduced the estimate \eqref{test1}, from \eqref{eee0} we have
$$
|\Delta^2 (\varphi^\delta) - \lambda\Delta \varphi^\delta| \leq \frac{C}{R^2} \varphi \quad\text{in } \R^N,
$$
and thus
\begin{equation}\label{tf1}
\begin{aligned}
\int_{B_{2R}}  (\Delta^2 u - \lambda \Delta u) \varphi^\delta =
\int_{B_{2R}}  u(\Delta^2 (\varphi^\delta) - \lambda\Delta \varphi^\delta) \leq  \frac{C}{R^2} \int_{B_{2R}}u\varphi.
\end{aligned}
\end{equation}
where $C > 0$ is a positive constant independent of $R$.
To estimate the right-hand side integral in \eqref{tf}, we first apply 
Lemma \ref{lm1}(i) to obtain  
$$    (\K1 * u^p)(x) \geq C|x|^{-\alpha}\log^{\beta}(1 + |x|)\quad\text{in } \R^N \setminus B_1.
$$
Hence, for $R>1$ large we have 
\begin{equation}\label{tf2}
\int_{B_{2R} }(\K1 * u^p)u^q\varphi^\delta \geq CR^{- \alpha}\log^{\beta}(1 + R) \int_{B_{2R}} u^q \varphi^\delta.
\end{equation}
Combining now \eqref{tf}, \eqref{tf1} and \eqref{tf2} we derive
\begin{equation}\label{tf3}
\int_{B_{2R}}u\varphi \geq CR^{2- \alpha}\log^{\beta}(1 + R) \int_{B_{2R}} u^q \varphi^\delta.
\end{equation}
By H\"older's inequality we obtain
\begin{equation*}
\int_{B_{2R}}u\varphi \leq \left(\int_{B_{2R}}u^q\varphi^\delta \right)^{\frac{1}{q}}\left(\int_{B_{2R}}\varphi^{\frac{q-\delta}{q-1}} \right)^{\frac{1}{q'}} \leq CR^{\frac{N}{q'}} \left(\int_{B_{2R}}u^q\varphi^{\delta} \right)^{\frac{1}{q}}.    
\end{equation*}
This implies
$$
\int_{B_{2R}}u^q\varphi^{\delta} \geq CR^{-N(q - 1)} \left(\int_{B_{2R}}u\varphi \right)^{q}.
$$
Plugging this last estimate into \eqref{tf3} we deduce
\begin{equation}\label{thm2vi2}
C'\geq R^{2 - \alpha - N(q - 1)}\log^{\beta}(1 + R) \left(\int_{B_{2R}}u\varphi \right)^{q-1}.
\end{equation}
Using $u(x)\geq c|x|^{2-N}$ in $\R^N\setminus B_1$ (which holds thanks to Lemma \ref{lm2}(ii)), it follows that
\begin{align*}
 C' &\geq CR^{2 - \alpha - N(q - 1)}\log^{\beta}(1 + R)\left(\int_{B_{2R}\setminus B_1}|x|^{2-N} dx \right)^{q - 1}\\
 &\geq CR^{2 - \alpha - N(q - 1) + 2(q - 1)}\log^{\beta}(1 + R) \nonumber\\
 &= CR^{N - \alpha - (N - 2)q}\log^{\beta}(1 + R)
 \quad\text{for all } R > 1,
\end{align*}
where $C, C' > 0$ are positive constants independent of $R$. Thus, we conclude that \eqref{P+} does not have non-negative solutions if $1 < q < \frac{N - \alpha}{N - 2}$ or $q = \frac{N - \alpha}{N - 2}$ and $\beta > 0$.\\
If $q = \frac{N - \alpha}{N - 2}$ and $-1+\frac{1}{q}<\beta\leq 0$, then $-\frac{\beta}{q-1}<1+\beta$ and we chose $\theta>0$ such that 
$$
0\leq -\frac{\beta}{q-1}<\theta<1+\beta<1.
$$
As in part (v), the estimate \eqref{est3} holds which we now use it into \eqref{thm2vi2} to obtain
\begin{align*}
 C' &\geq CR^{2 - \alpha - N(q - 1)}\log^{\beta}(1 + R)\left(\int_{B_{2R}\setminus B_1}|x|^{2-N}\log^\theta|x|  dx \right)^{q - 1}\\
 &\geq CR^{2 - \alpha - N(q - 1) + 2(q - 1)}\log^{\beta+(q-1)\theta}R \nonumber\\
 &= C\log^{\beta+(q-1)\theta}R
 \quad\text{for all } R > 1.
\end{align*}
This leads to a contradiction as $R\to \infty$, since $\beta+(q-1)\theta>0$.\\
\\
(viii) Assume $p=\frac{N}{N-2}$, $q = \frac{N - \alpha}{N - 2}$ and $-2 + \frac{1}{q} < \beta \leq 0$. Then, using Lemma \ref{lm2}(ii) we have
$$
u^p(x)\geq c|x|^{-p(N-2)}=c|x|^{-N}\quad\mbox{  in } \R^N\setminus B_1.
$$ 
By Lemma \ref{lm1}(i) it follows that $\K1 * u^p$ satisfies the improved estimate
\begin{equation}\label{lmr8}
(\K1 * u^p)(x) \geq C|x|^{-\alpha}\log^{1 + \beta}(1 + |x|)\quad\text{in }\R^N \setminus B_1.
\end{equation}
Therefore, retaking the estimates \eqref{tf}-\eqref{tf3} in the proof of part (vi-vii) and \eqref{lmr8}, for all $R > 1$ we have
\begin{align}\label{lmr9}
\int_{B_{2R}}u\varphi &\geq CR^{2}\int_{B_{2R}}\varphi^\delta (\Delta^2 u - \lambda \Delta u) \nonumber\\
&\geq CR^{2}\int_{B_{2R}}(\K1 * u^p)u^q\varphi^\delta  \nonumber \\
&\geq CR^{2 - \alpha}\log^{1 + \beta}(1 + R)\int_{B_{2R}}u^q\varphi^\delta  \nonumber\\
&\geq CR^{2 - \alpha - N(q - 1)}\log^{1 + \beta}(1 + R)\left(\int_{B_{2R}}u\varphi \right)^{q}.
\end{align}
We next divide our analysis into two separate cases.\\
{\bf Case 1:} If $-1 < \beta \leq 0$. Since $u(x) \geq c|x|^{2 - N}$ in $B_{2R}\setminus B_1$, it follows that
\begin{align*}
 C' &\geq CR^{2 - \alpha - N(q - 1)}\log^{1 + \beta}(1 + R)\left(\int_{B_{2R}}u\varphi \right)^{q - 1}\\
 &\geq CR^{2 - \alpha - N(q - 1) + 2(q - 1)}\log^{1 + \beta}R \nonumber\\
 &= C\log^{1 + \beta}R
 \quad\text{for all } R > 1.   
\end{align*}
where $C, C' > 0$ are positive constants independent of $R$. Thus, since $1 + \beta > 0$, we conclude that \eqref{P+} does not have non-negative solutions in Case 1.\\
{\bf Case 2:} If $-2 + \frac{1}{q}<\beta \leq -1$. We choose $\theta_0 > 0$ such that
\begin{equation}\label{thet}
0 < \frac{-1 - \beta}{q - 1} < \theta_0 < 2 + \beta \leq  1.
\end{equation}
Using the estimate \eqref{est3} into \eqref{lmr9} we obtain:
\begin{align*}
 C' &\geq CR^{2 - \alpha - N(q - 1)}\log^{1 + \beta}(1 + R)\left(\int_{B_{2R}}|x|^{2-N}\log^{\theta_0}|x| \right)^{q - 1}\\
 &\geq CR^{2 - \alpha - N(q - 1) + 2(q - 1)}\log^{1 + \beta + (q - 1)\theta_0}R \nonumber\\
 &= C\log^{1 + \beta+ (q - 1)\theta_0}R
 \quad\text{for all $R$ large}.    
\end{align*}
where $C, C' > 0$ are positive constants independent of $R$. Thus, since $1 + \beta+ (q - 1)\theta_0>0$, we conclude that \eqref{P+} does not have non-negative solutions in Case 2, thereby completes our proof by combining Case 1 and Case 2.\\
\\
(ix) Assume $p=\frac{N - \alpha}{N-2}$, $q = \frac{N}{N - 2}$ and $-2 + \frac{1}{q}<\beta \leq 0$. Let us note that from part (iii) above, \eqref{P+} has no solutions if $\beta \geq -1$. Thus, we shall assume in the following that $-2 + \frac{1}{q}<\beta \leq -1$. Using Lemma \ref{lm2}(ii) we have
$$
u^p(x)\geq c|x|^{-p(N-2)}=c|x|^{-(N - \alpha)}\quad\mbox{  in } \R^N\setminus B_1.
$$ 
Thus, for all $|y| \geq 2|x| > 2$ and $\frac{|y|}{2} \leq |x - y| \leq \frac{3|y|}{2}$, one has
\begin{align}\label{lm291}
(\K1 * u^p)(x) &\geq C\int_{|y| \geq 2|x|}|y|^{-\alpha - p(N - 2)}\log^{\beta}(1 + |y|)dy \nonumber\\[0.2cm]
&= C\int_{|y| \geq 2|x|}|y|^{-N}\log^{\beta}(1 + |y|)dy \nonumber\\[0.2cm]
&\geq C\log^{1 + \beta}(1 + |x|)\quad\text{for all }x \in \R^N\setminus B_1.
\end{align}
Retaking the estimates \eqref{tf}-\eqref{tf3} in the proof of part (vi-vii) and \eqref{lm291}, for all $R > 1$ we have
\begin{align}\label{lm292}
\int_{B_{2R}}u\varphi &\geq CR^{2}\int_{B_{2R}} (\Delta^2 u - \lambda \Delta u)\varphi^\delta \nonumber\\
&\geq CR^{2}\int_{B_{2R}}(\K1 * u^p)u^q\varphi^\delta  \nonumber \\
&\geq CR^{2}\log^{1 + \beta}(1 + R)\int_{B_{2R}}u^q\varphi^\delta  \nonumber\\
&\geq CR^{2 - N(q - 1)}\log^{1 + \beta}(1 + R)\left(\int_{B_{2R}}u\varphi \right)^{q}.
\end{align}
Similarly, we choose $\theta_0 > 0$ that satisfies \eqref{thet}.
Using the estimate \eqref{est3} into \eqref{lm292} we find:
\begin{align*}
 C' &\geq CR^{2 - N(q - 1)}\log^{1 + \beta}(1 + R)\left(\int_{B_{2R}}|x|^{2-N}\log^{\theta_0}|x| \right)^{q - 1}\\
 &\geq CR^{2 - N(q - 1) + 2(q - 1)}\log^{1 + \beta + (q - 1)\theta_0}R \nonumber\\
 &= C\log^{1 + \beta+ (q - 1)\theta_0}R
 \quad\text{for all $R$ large}.   
\end{align*}
where $C, C' > 0$ are positive constants independent of $R$. Thus,  since $1 + \beta+ (q - 1)\theta_0>0$, we conclude that \eqref{P+} does not have non-negative solutions in this case and thus complete the proof of our Theorem.\qed
\section{Proof of Theorem \ref{thm3}} 
Assume $N \geq 3$. To construct a solution $u \in \mathcal{C}^4(\R^N)$, we consider 
\begin{equation}\label{Pu}
    u(x) = \Phi(x) * w^{-\gamma}\log^{\tau}w \quad\text{in } \R^N,
\end{equation}
where $w(x) = (A + |x|^2)^{\frac{1}{2}}$ for some $A > e$ and $\Phi(x) = \frac{1}{(N - 2)\sigma_N}|x|^{-(N - 2)}$ is the fundamental solution of the Laplace operator.
\begin{lemma}\label{lml}
    Assume $2 < \gamma \leq N$, $-1 < \tau < 1$ and $A > e$. Let $u$ be given by \eqref{Pu}. Then, there exists $\lambda^* = \lambda(N, A, \gamma, \tau) > 0$ so that for all $\lambda > \lambda^*$ large one has
    \begin{equation*}
     \Delta^2 u - \frac{\lambda}{2}\Delta u \geq 0 \quad\text{in } \R^N. 
    \end{equation*}
    In  particular, one has
    \begin{equation}\label{dPl}
     \Delta^2u - \lambda\Delta u \geq -\frac{\lambda}{2} \Delta u = \frac{\lambda}{2}w^{-\gamma}\log^{\tau} w\quad \text{in }\R^N.   
     \end{equation}
\end{lemma}
\begin{proof}
From \eqref{Pu} we have
\begin{equation}\label{dPu}
-\Delta u(x) = w(x)^{-\gamma}\log^{\tau} w(x) \quad\text{in } \R^N.
\end{equation}
Let $v = -\Delta u$, by computing directly, we have
\begin{align*}
    v' =\, &-\gamma rw^{-\gamma - 2} \log^{\tau} w + \tau rw^{-\gamma - 2}\log^{\tau-1}w,\\
    v'' =\, &-\gamma w^{-\gamma - 2}\log^{\tau} w + \tau w^{-\gamma - 2}\log^{\tau-1} w + \gamma(\gamma + 2) r^2 w^{-\gamma - 4} \log^{\tau} w \\
    &- \tau(2\gamma + 2) r^2 w^{-\gamma - 4}\log^{\tau-1} w + \tau (\tau - 1)r^2 w^{-\gamma - 4}\log^{\tau - 2} w.     
\end{align*}
Since $r^2 = w^2 - A$, we find
\begin{align}\label{dPuu}
    -&\Delta v = -v'' - \frac{N - 1}{r}v' \nonumber\\
    =& \gamma(N - \gamma - 2)w^{-\gamma - 2}\log^{\tau}w + \tau (2\gamma + 2 - N)w^{-\gamma - 2}\log^{\tau -1}w + A\gamma(\gamma + 2)w^{-\gamma - 4}\log^{\tau}w \nonumber\\
     &- A\tau (2\gamma + 2)w^{-\gamma - 4}\log^{\tau -1}w + \tau(1 - \tau)(w^2  - A)w^{-\gamma - 4}\log^{\tau -2}w.
\end{align}
Then, combining \eqref{dPl}, \eqref{dPu} and \eqref{dPuu}, we get
\begin{align*}\label{dPl1}
&\Delta^2 u - \frac{\lambda}{2} \Delta u = -\Delta v + \frac{\lambda}{2} v \nonumber\\
=& w^{-\gamma - 2}\log^{\tau -1}w\left\{\left[\gamma(N - \gamma - 2) + \frac{\lambda w^2}{2}\right]\log w + \tau (2\gamma + 2 - N) \right\} \nonumber\\
&+Aw^{-\gamma - 4}\log^{\tau -1}w\left\{\gamma(\gamma + 2)\log w - \tau (2\gamma + 2) \right\}  + \tau(1 - \tau)(w^2  - A)w^{-\gamma - 4}\log^{\tau -2}w.
\end{align*}
Hence,
\begin{align*}
\frac{\Delta^2 u - \frac{\lambda}{2} \Delta u}{w^{-\gamma - 2}\log^{\tau -1}w} =&
\left[\gamma(N - \gamma - 2) + \frac{\lambda w^2}{2}\right]\log w + \tau (2\gamma + 2 - N) \nonumber\\
&+A\frac{\gamma(\gamma + 2)\log w - \tau (2\gamma + 2)}{w^2}  + \frac{\tau(1 - \tau)}{\log w}\Big(1  - \frac{A}{w^2}\Big).
\end{align*}
The leading term in the right-hand side  of the above equality is $\frac{\lambda w^2}{2}\log w$ and it is the only term containing $\lambda$. Thus, we may take the proper constant $\lambda^* = \lambda(N, A, \gamma, \tau) > 0$ such that
\begin{equation*}
\Delta^2 u - \frac{\lambda}{2}\Delta u \geq 0 \quad\text{in } \R^N. 
\end{equation*}
From the above, the inequality \eqref{dPl} follows immediately.
\end{proof}\vspace{-0.2cm}
\begin{lemma}\label{lmr}
Assume $0 \leq \alpha < N$, $\beta > \alpha - N$, $p, q > 0$, $2 < \gamma \leq N$ and $-1 < \tau < 1$. If $u(x)$ is the function defined in \eqref{Pu}, then for all $x \in \R^N$ one has
\begin{equation*}
u(x) \leq c \left\{
    \begin{aligned}
       &w^{-(\gamma - 2)}\log^{\tau}w &\text{if }& 2 < \gamma < N,\\
       &w^{-(N- 2)}\log^{1 + \tau}w &\text{if }& \gamma = N,
    \end{aligned}
\right.
\end{equation*}
for some constant $c>0$.
Furthermore, there exists $C = C(N, \alpha, \beta, A, p, q, \gamma, \tau) > 0$ such that
\begin{equation}\label{l52}
\begin{aligned}
(\K1 * u^p)u^q(x) &\leq C \begin{cases}
       \ds w^{-(\gamma - 2)q}\log^{1 + \beta + \tau(p + q)}w&\text{if }\ds \gamma < N,\, p = \frac{N - \alpha}{\gamma -2},\, \beta + \tau p < -1,\\[0.3cm] 
       \ds w^{N - \alpha - (\gamma - 2)(p + q)}\log^{\beta + \tau (p + q)}w &\text{if }\ds \gamma < N,\, \frac{N - \alpha}{\gamma - 2} < p < \frac{N}{\gamma - 2},\\[0.3cm]
       w^{- \alpha - (\gamma - 2)q}\log^{1 + \beta + \tau (p + q)}w &\text{if }\ds \gamma < N,\, p = \frac{N}{\gamma - 2},\,  \tau p>-1\\[0.3cm] 
       w^{- \alpha - (\gamma - 2)q}\log^{\beta + \tau q}w &\text{if }\ds \gamma < N,\, p > \frac{N}{\gamma - 2},\\[0.3cm] 
       w^{-(N - 2)q}\log^{1 + \beta + (1 + \tau)(p + q)}w&\text{if }\ds \gamma = N,\, p = \frac{N - \alpha}{N -2},\, \beta + (1 + \tau) p < -1,\\[0.3cm] 
       w^{N - \alpha - (N - 2)(p + q)}\log^{\beta + (1 + \tau)(p + q)}w&\text{if }\ds \gamma = N,\, \frac{N - \alpha}{N - 2} < p < \frac{N}{N - 2},
       \\[0.3cm] 
       w^{- \alpha - (N - 2)q}\log^{1 + \beta + (1 + \tau)(p + q)}w &\text{if }\ds \gamma = N,\, p = \frac{N}{N - 2},\\[0.3cm] 
       w^{- \alpha - (N - 2)q}\log^{\beta + (1 + \tau)q}w &\text{if }\ds \gamma = N,\, p > \frac{N}{N - 2}.
    \end{cases}
\end{aligned}
\end{equation}
\end{lemma}
\begin{proof}
Recall from \eqref{ptt} that
\[\K1(x) = |x|^{-\alpha}\log^{\beta}(1 + |x|)\quad\text{in }\R^N,\]
where $0 \leq \alpha \leq N, \beta > \alpha - N$. From \eqref{Pu}, one has
\[u = K_{N-2,0} * w^{-\gamma}\log^{\tau}w,\]
where $2 < \gamma \leq N$, $-1 < \tau < 1$ and $w(x) = (A + |x|^2)^{\frac{1}{2}}$ with $A > e$ large. We first apply Lemma \ref{lm1}(iii) for $\alpha = N - 2$, $\beta = 0$ and $f = w^{-\gamma}\log^{\tau}w$ so that $\sigma = \gamma$, $\kappa=\tau$. Thus, for all $x \in \R^N$, $u(x)$ satisfies the following estimate:
\begin{equation}\label{lmr0}
u(x) \leq c \left\{
    \begin{aligned}
       &w^{-(\gamma - 2)}\log^{\tau}w &\text{if }&2 < \gamma  < N\\
       &w^{-(N - 2)}\log^{1 + \tau}w &\text{if }& \gamma = N
    \end{aligned}
\right.\quad\text{ in }\R^N.
\end{equation}
In the case $2 < \gamma < N$ one has
\[u(x) \leq cw^{-(\gamma - 2)}\log^{\tau}w \quad\text{in }\R^N.\]
We next apply Lemma \ref{lm1}(iii) again for $f = u^p \leq cw^{-(\gamma - 2)p}\log^{\tau p}w$ so that $\sigma = (\gamma - 2)p$, $\kappa = \tau p$. Then, for all $x \in \R^N$ one has
\begin{align}\label{gln}
(\K1 * u^p)(x) \leq c \left\{
    \begin{aligned}
       &\log^{1 + \beta + \tau p}w&\text{if }&(\gamma - 2)p = N - \alpha,\, \beta + \tau p < -1,\\
       &w^{N - \alpha - (\gamma - 2)p}\log^{\beta + \tau p}w &\text{if }&N  - \alpha < (\gamma - 2)p < N,\\
       &w^{- \alpha}\log^{1 + \beta + \tau p}w &\text{if }&(\gamma - 2)p = N,\, \tau p > -1,\\
       &w^{- \alpha}\log^{\beta}w &\text{if }&(\gamma - 2)p > N.\\
    \end{aligned}
\right.
\end{align}
Similarly, in the case $\gamma = N$ one has
\[u(x) \leq cw^{-(N - 2)}\log^{1 + \tau}w\quad\text{in }\R^N.\]
We apply Lemma \ref{lm1}(iii) for $f = u^p \leq cw^{-(N - 2)p}\log^{(1 + \tau)p}w$ so that $\sigma = (N - 2)p$, $\kappa = (1 + \tau)p$. Thus, for all $x \in \R^N$, we have
\begin{align}\label{gen}
(\K1 * u^p)(x) 
&\leq c \left\{
    \begin{aligned}
        &\log^{1 + \beta + (1 + \tau) p}w&\text{if }&(N - 2)p = N - \alpha,\, \beta + (1 + \tau) p < -1,\\
       &w^{N - \alpha - (N - 2)p}\log^{\beta + (1 + \tau)p}w &\text{if }&N - \alpha < (N - 2)p < N,\\
       &w^{- \alpha}\log^{1 + \beta + (1 + \tau)p}w &\text{if }&(N - 2)p = N,\\
       &w^{- \alpha}\log^{\beta}w &\text{if }&(N - 2)p > N.\\
    \end{aligned}
\right.
\end{align}
Hence, we combine \eqref{gln}-\eqref{gen} to obtain
\begin{align*}
(\K1 * u^p)(x)\leq C \left\{
    \begin{aligned}
       &\log^{1 + \beta + \tau p}w&\text{if }&\gamma < N,\, p = \frac{N - \alpha}{\gamma -2},\, \beta + \tau p < -1,\\ 
       &w^{N - \alpha - (\gamma - 2)p}\log^{\beta + \tau p}w &\text{if }&\gamma < N,\, \frac{N - \alpha}{\gamma - 2} < p < \frac{N}{\gamma - 2},\\
       &w^{- \alpha}\log^{1 + \beta +\tau p}w &\text{if }&\gamma < N,\, p = \frac{N}{\gamma - 2},\, \tau p>-1\\
       &w^{- \alpha}\log^{\beta}w &\text{if }&\gamma < N,\, p > \frac{N}{\gamma - 2},\\
       &\log^{1 + \beta + (1 + \tau) p}w&\text{if }&\gamma = N,\, p = \frac{N - \alpha}{N -2},\, \beta + (1 + \tau) p < -1,\\ 
       &w^{N - \alpha - (N - 2)p}\log^{\beta + (1 + \tau)p}w &\text{if }&\gamma = N,\, \frac{N - \alpha}{N - 2} < p < \frac{N}{N - 2},\\
       &w^{- \alpha}\log^{1 + \beta + (1 + \tau)p}w &\text{if }&\gamma = N,\, p = \frac{N}{N - 2},\\
       &w^{- \alpha}\log^{\beta}w &\text{if }&\gamma = N,\, p > \frac{N}{N - 2}.
    \end{aligned}
\right.
\end{align*}
From the above estimates and \eqref{lmr0} we deduce \eqref{l52}.
\vspace{-0.2cm}
\end{proof}
\begin{lemma}\label{lmr2}
    Assume $\alpha = N$, $\beta > 0$, $p,\,q >0$, $2 < \gamma \leq N$ and $\tau = 0$. If $u(x)$ is the function defined in \eqref{Pu}, then for all $x \in \R^N$ one has
\begin{equation*}
u(x) \leq c \left\{
    \begin{aligned}
       &w^{-(\gamma - 2)} &\text{if }& 2 < \gamma < N,\\
       &w^{-(N- 2)}\log w &\text{if }& \gamma = N,
    \end{aligned}
\right.
\end{equation*}
for some constant $c>0$. Furthermore, Thus, there exists $C = C(N, \beta, A, p, q, \gamma) > 0$ such that
\begin{equation}\label{l53}
(\K1 * u^p)u^q(x)
\leq C \begin{cases}
           w^{- (\gamma - 2)(p + q)}\log^{1 + \beta}w &\text{if }\gamma < N,\, 0 < p < \frac{N}{\gamma - 2},\\[0.2cm]
       w^{- N - (N - 2)q}\log^{\beta + q}w &\text{if }\gamma = N,\, p > \frac{N}{N - 2}.
    \end{cases}
\end{equation}
\vspace{-0.2cm}
\end{lemma}\label{lmrN}
\begin{proof}
We first apply a similar argument to that in \eqref{lmr0} with $\tau = 0$ to obtain the following estimates.
\begin{equation}\label{lmrN0}
u(x) \leq c \left\{
    \begin{aligned}
       &w^{-(\gamma - 2)} &\text{if }&2 < \gamma  < N\\
       &w^{-(N - 2)}\log w &\text{if }& \gamma = N
    \end{aligned}
\right.\quad\text{ in }\R^N.
\end{equation}
Since $\alpha = N$ and $\beta > 0$, one has
\[K_{N, \beta}(x) = |x|^{-N}\log^{\beta}(1 + |x|)  \quad\text{in }\R^N.\]
We apply Lemma \ref{lm1}(iii) for $K_{N, \beta}$ and $f = u^p$, where $u(x)$ satisfies \eqref{lmrN0}. Thus, for all $x \in \R^N$ one has
\begin{align}\label{lmrN1}
(K_{N, \beta} * u^p)(x) \leq c \left\{
    \begin{aligned}
       &w^{ - (\gamma - 2)p}\log^{1 + \beta} w &\text{if }&0 < \gamma < N, 0 < (\gamma - 2)p < N,\\
       &w^{- N}\log^{\beta}w &\text{if }&\gamma  = N, (N - 2)p > N.\\
    \end{aligned}
\right.
\end{align}
Combining now \eqref{lmrN0} and \eqref{lmrN1} we deduce \eqref{l53}.
\vspace{-0.2cm}
\end{proof}
We are now ready to proceed with the construction of a positive solution to \eqref{P+} for varying cases regarding the value of $p, q, p + q$ and $\beta$.
\vspace{-0.2cm}
\begin{flalign*}
\textbf{Case 1:}\quad p > \frac{N - \alpha}{N - 2},\quad q > \frac{N - \alpha}{N - 2},\quad p + q > \frac{2N - \alpha}{N - 2}\quad\text{and}\quad \beta > \alpha - N.& &   
\end{flalign*}
We divide our analysis  into two separate cases.\\
\medskip
\noindent{\bf Case 1a:} $\frac{N - \alpha}{N - 2} < p \leq \frac{N}{N - 2}$. This means $N - \alpha < (N - 2)p \leq N$ and we also have $(N - 2)(p + q) > 2N - \alpha$. Thus, we choose $\gamma < N$ close to $N$ so that
\begin{equation}\label{P+1a1}
N>    (\gamma - 2)p > N - \alpha\quad\text{and}\quad (\gamma - 2)(p + q) > 2N - \alpha > N + \gamma - \alpha.
\end{equation}
Take $u$ as defined in \eqref{Pu}, with $\gamma$ satisfying \eqref{P+1a1} and $\tau = 0$. Then, by the second estimate in \eqref{l52} one has
\begin{align}\label{P+1a}
(\K1 * u^p)u^q(x) \leq c w^{N - \alpha - (\gamma - 2)(p + q)}\log^{\beta}w \leq c w^{-\gamma} \quad\text{in }\R^N.
\end{align}
Combining \eqref{P+1a} with \eqref{dPl} in Lemma \ref{lml} we find
\begin{equation*}
    \Delta^2 u - \lambda \Delta u \geq \frac{\lambda}{2}w^{-\gamma} \geq C(\K1 * u^p)u^q \quad\text{in }\R^N.
\end{equation*}
\medskip

\noindent{\bf Case 1b:} $p > \frac{N}{N - 2}$. Let $u$ be given by \eqref{Pu} where $\gamma = N$ and $\tau = 0$. Since $p > \frac{N}{N - 2}$ and $q > \frac{N - \alpha}{N - 2}$, we apply the last estimate in \eqref{l52} to obtain
\begin{align}\label{P+1b}
(\K1 * u^p)u^q(x) \leq c w^{- \alpha - (N - 2)q}\log^{\beta + q}w \leq c w^{-N}.
\end{align}
Combining \eqref{P+1b} with \eqref{dPl} in Lemma \ref{lml} we find
\begin{equation*}
    \Delta^2 u - \lambda \Delta u \geq \frac{\lambda}{2}w^{-N} \geq C(\K1 * u^p)u^q \quad\text{in }\R^N.
\end{equation*}
Thus, for suitable constants $C = C(N, \alpha, \beta, \lambda, p, q, A) > 0$ in the case 1a and 1b separately, $Cu$ is a solution of \eqref{P+}.\\
\vspace{-0.2cm}
\begin{flalign*}
\textbf{Case 2:}\quad p = \frac{N - \alpha}{N - 2},\quad q > \frac{N}{N - 2}\quad\text{and}\quad \beta < -1.& & 
\end{flalign*}
Let $u$ be given by \eqref{Pu} where $\gamma = N$ and $\tau > -1$ close to $-1$ so that
\[\beta + (1 + \tau)p < -1.\]
Since $p = \frac{N - \alpha}{N - 2}$, by applying the fifth estimate in \eqref{l52} one has
\begin{align}\label{P+2}
(\K1 * u^p)u^q(x) \leq c w^{ - (N - 2)q}\log^{1 + \beta + (1 + \tau)(p + q)}w \leq c w^{-N}\log^{\tau}w\quad\text{in }\R^N.
\end{align}
Combining \eqref{P+2} with \eqref{dPl} in Lemma \ref{lml} we find
\begin{equation*}
    \Delta^2 u - \lambda \Delta u \geq \frac{\lambda}{2}w^{-N}\log^{\tau}w \geq C(\K1 * u^p)u^q \quad\text{in }\R^N.
\end{equation*}
Thus, for a suitable constant $C = C(N, \alpha, \beta, \lambda, p, q, A) > 0$, $Cu$ is a solution of \eqref{P+}.\\
\vspace{-0.2cm}
\begin{flalign*}
\textbf{Case 3:}\quad p > \frac{N}{N - 2},\quad q = \frac{N - \alpha}{N - 2}\quad\text{and}\quad \beta < -1.& & 
\end{flalign*}
Let $u$ be given by \eqref{Pu} where $\gamma = N$ and $\tau > -1$ close to $-1$ so that
\[\tau > \beta + (1 + \tau)q.\]
Since $p > \frac{N}{N - 2}$ and $q = \frac{N - \alpha}{N - 2}$, by the last estimate in \eqref{l52} one has
\begin{align}\label{P+3}
(\K1 * u^p)u^q(x) \leq c w^{- \alpha - (N - 2)q}\log^{\beta + (1 + \tau)q}w \leq c w^{-N}\log^{\tau}w\quad\text{in }\R^N.
\end{align}
Combining \eqref{P+3} with \eqref{dPl} in Lemma \ref{lml} we find
\begin{equation*}
    \Delta^2 u - \lambda \Delta u \geq \frac{\lambda}{2}w^{-N}\log^{\tau}w \geq C(\K1 * u^p)u^q \quad\text{in }\R^N.
\end{equation*}
Thus, for a suitable constant $C = C(N, \alpha, \beta, \lambda, p, q, A) > 0$, $Cu$ is a solution of \eqref{P+}.\\
\vspace{-0.2cm}
\begin{flalign*}
\textbf{Case 4:}\quad p > \frac{N - \alpha}{N - 2},\quad q > \frac{N - \alpha}{N - 2},\quad p + q = \frac{2N - \alpha}{N - 2}\quad\text{and} \quad\beta < -1.& &   
\end{flalign*}
Let $u$ be given by \eqref{Pu} where $\gamma = N$ and $\tau > -1$ close to $-1$ so that
\[\tau > \beta + (1 + \tau)(p + q).\]
Since $p, q > \frac{N -  \alpha}{N - 2}$ and $p + q = \frac{2N - \alpha}{N - 2}$, we have  $N - \alpha < (N - 2)p < N$. Then, by the sixth estimate in \eqref{l52}  one has
\begin{align}\label{P+4}
(\K1 * u^p)u^q(x) \leq c w^{N -\alpha - (N - 2)(p + q)}\log^{\beta + (1 + \tau)(p + q)}w \leq c w^{-N}\log^{\tau}w.
\end{align}
Combining \eqref{P+4} with \eqref{dPl} in Lemma \ref{lml} we find
\begin{equation*}
    \Delta^2 u - \lambda \Delta u \geq \frac{\lambda}{2}w^{-N}\log^{\tau}w \geq C(\K1 * u^p)u^q \quad\text{in }\R^N.
\end{equation*}
Thus, for a suitable constant $C = C(N, \alpha, \beta, \lambda, p, q, A) > 0$, $Cu$ is a solution of \eqref{P+}.\\
\vspace{-0.2cm}
\begin{flalign*}
\textbf{Case 5:}\quad p = \frac{N - \alpha}{N - 2},\quad q = \frac{N}{N - 2}\quad\text{and}\quad \beta < -2.& &   
\end{flalign*}
Let $u$ be given by \eqref{Pu} where $\gamma = N$ and $\tau > -1$ close to $-1$ so that
\[\tau > 1 + \beta + (1 + \tau)(p + q).\]
Since $p = \frac{N - \alpha}{N - 2}$, by the fifth estimate in \eqref{l52} one has
\begin{align}\label{P+5}
(\K1 * u^p)u^q(x) \leq c w^{- (N - 2)q} \log^{1 + \beta + (1 + \tau)(p + q)}w \leq c w^{-N}\log^{\tau}w.
\end{align}
Combining \eqref{P+5} with \eqref{dPl} in Lemma \ref{lml} we find
\begin{equation*}
    \Delta^2 u - \lambda \Delta u \geq \frac{\lambda}{2}w^{-N}\log^{\tau}w \geq C(\K1 * u^p)u^q \quad\text{in }\R^N.
\end{equation*}
Thus, for a suitable constant $C = C(N, \alpha, \beta, \lambda, p, q, A) > 0$, $Cu$ is a solution of \eqref{P+}.\\
\vspace{-0.2cm}
\begin{flalign*}
\textbf{Case 6:}\quad p = \frac{N}{N - 2},\quad q = \frac{N - \alpha}{N - 2}\quad\text{and}\quad \beta < -2.& & 
\end{flalign*}
Let $u$ be given by \eqref{Pu} where $\gamma = N$ and $\tau > -1$ close to $-1$ so that
\[\tau > 1 + \beta + (1 + \tau)(p + q).\]
Since $p = \frac{N}{N - 2}$, by the seventh estimate in \eqref{l52} one has
\begin{align}\label{P+6}
(\K1 * u^p)u^q(x) \leq c w^{-\alpha - (N - 2)q}\log^{1 + \beta + (1 + \tau)(p + q)}w \leq c w^{-N}\log^{\tau}w.
\end{align}
Combining \eqref{P+6} with \eqref{dPl} in Lemma \ref{lml} we find
\begin{equation*}
    \Delta^2 u - \lambda \Delta u \geq \frac{\lambda}{2}w^{-N}\log^{\tau}w \geq C(\K1 * u^p)u^q \quad\text{in }\R^N.
\end{equation*}
Thus, for a suitable constant $C = C(N, \alpha, \beta, \lambda, p, q, A) > 0$, $Cu$ is a solution of \eqref{P+}.
\qed
\section{Proof of Theorem \ref{thm4}} 
Assume $N \geq 3$, $\alpha = N$, $\beta > 0$ and $q > 0$. Since $p \geq 1$, we split our proof into two cases: $1 \leq p \leq \frac{N}{N - 2}$ and $p > \frac{N}{N - 2}$.\\
\noindent{\bf Case 1:} $1 \leq p \leq \frac{N}{N - 2}$. This means $N - 2 \leq (N - 2)p \leq N$ and we also have $(N - 2)(p + q) > N$. Thus, we choose $\gamma < N$ close to $N$ so that 
\begin{equation}\label{thm41}
    (\gamma - 2)p < N \quad\text{and}\quad (\gamma - 2)(p + q) >N> \gamma.
\end{equation}
Let $u$ be the function defined in \eqref{Pu}, with $\gamma$ satisfying \eqref{thm41} and $\tau = 0$. Then, by the first estimate in \eqref{l53} one has
\begin{equation}\label{thm42}
   (\K1 * u^p)u^q(x) \leq w^{- (\gamma - 2)(p + q)}\log^{1+\beta}w \leq c w^{-\gamma} \quad\text{in }\R^N. 
\end{equation}
Combining \eqref{thm42} with \eqref{dPl} in Lemma \ref{lml} we find
\begin{equation*}
    \Delta^2 u - \lambda \Delta u \geq \frac{\lambda}{2}w^{-\gamma} \geq C(\K1 * u^p)u^q \quad\text{in }\R^N.
\end{equation*}
\noindent{\bf Case 2:} $p > \frac{N}{N - 2}$. Let $u$ be given by \eqref{Pu} where $\gamma = N$ and $\tau = 0$. Since $p > \frac{N}{N - 2}$ and $q > 0$, we apply the second estimate in \eqref{l53} to obtain
\begin{equation}\label{thm43}
       (\K1 * u^p)u^q(x) \leq w^{-N - (N - 2)q}\log^{\beta + q}w \leq c w^{-N} \quad\text{in }\R^N.
\end{equation}
Combining \eqref{thm43} with \eqref{dPl} in Lemma \ref{lml} we find
\begin{equation*}
    \Delta^2 u - \lambda \Delta u \geq \frac{\lambda}{2}w^{-N} \geq C(\K1 * u^p)u^q \quad\text{in }\R^N.
\end{equation*}
Thus, for suitable constants $C = C(N, \beta, \lambda, p, q, A) > 0$ in the Case 1 and Case 2 separately, $Cu$ is a solution of \eqref{P+}.
\qed

\section*{Author contributions} 

Zhe Yu has conducted the research included in this paper and prepared the manuscript during the submission process.

\section*{Data availability}
No datasets were generated or analyzed during the current research.
\section*{Declarations}
\section*{Conflict of interest} 
The author declares no conflict of interest.

\section*{Competing interests} The author declares no competing interests.

\end{document}